\documentclass[11pt]{amsart}
\usepackage[margin=1.2in]{geometry}
\usepackage{amsmath}
\usepackage{color}
\usepackage{amsthm}
\usepackage{amsfonts}
\usepackage{amssymb}
\usepackage{amscd}
\numberwithin{equation}{section}
\newtheorem{theorem}{Theorem}[section]
\newtheorem{cor}[theorem]{Corollary}
\newtheorem{proposition}[theorem]{Proposition}
\newtheorem{lemma}[theorem]{Lemma}
\newtheorem{prop}[theorem]{Proposition}
\theoremstyle{definition}
\newtheorem{definition}[theorem]{Definition}
\newtheorem{remark}[theorem]{Remark}
\newtheorem{example}[theorem]{Example}

\newcommand\half{\frac{1}{2}}

\DeclareMathOperator{\ad}{ad}

\newcommand\be{\beta}

\newcommand\g{\mathfrak g}

\newcommand\h{\mathfrak h}

\newcommand\D{\Delta}
\renewcommand\l{\lambda}

\renewcommand\d{\delta}

\renewcommand\a{\alpha}

\renewcommand\k{\mathfrak k}

\newcommand{\Z}{\mathbb Z}
\newcommand\nat{\mathbb N}
\newcommand\ganz{\mathbb Z}

\newcommand\s{\sigma}

\newcommand\C{\mathbb C}
\newcommand\R{\mathbb R}

\newcommand\G{\Gamma}

\newcommand{\fg}{\mathfrak{g}}

\newcommand{\ZZ}{\mathbb{Z}}

\newcommand{\sdim}{\text{\rm sdim}}

\newcommand{\Cur}{\mbox{Cur}\,}
\newcommand{\End}{\mbox{End}}

\newcommand{\Hom}{\mbox{Hom}}

\newcommand{\Res}{\mbox{Res}}

\newcommand{\vac}{{\bf 1}}
\newcommand{\bea}{\begin{eqnarray}}
\newcommand{\eea}{\end{eqnarray}}
\begin{document}
\title{Invariant Hermitian forms on vertex algebras}
\author[Victor~G. Kac, Pierluigi M\"oseneder Frajria,  Paolo  Papi]{Victor~G. Kac\\Pierluigi M\"oseneder Frajria\\Paolo  Papi}
\begin{abstract} We study invariant Hermitian forms on a conformal vertex algebra and on their (twisted) modules. We establish existence of a non-zero invariant Hermitian form on an arbitrary $W$--algebra. We show that for a minimal simple $W$--algebra $W_k(\g,\theta/2)$ this form can be unitary only when its $\tfrac{1}{2}\Z$--grading is compatible with parity, unless $W_k(\g,\theta/2)$ ``collapses'' to its affine subalgebra.
\end{abstract}
\maketitle
\section{Introduction}
In the present paper we study invariant Hermitian forms on a conformal vertex algebra $V$  and its  (possibly twisted) positive energy modules. By a conformal vertex algebra we mean a vector superspace $V$ over $\C$, endowed with a structure of a vertex algebra (with state--field correspondence $a\mapsto Y(a,z)$), and a Virasoro vector $L$ such that the eigenvalues of $L_0$ lie in $\tfrac{1}{2}\Z_+$, all eigenspaces are finite--dimensional, and the $0$--th eigenspace consists of multiples of the vacuum vector (cf. Definition 1.1 in Section \ref{Section1} and \cite{KB}).

Let $\phi$ be a conjugate linear involution of $V$. A Hermitian form $(\, \cdot \,\, , \, \cdot\, )$ on $V$ is called $\phi$--{\it invariant} if, for all $a\in V$, one has 
\begin{equation}\label{111}
(v,Y(a,z)u)=(Y(A(z)a,z^{-1})v,u),\quad u,v\in V.
\end{equation}
Here $A(z):V\to V((z))$ is defined by 
\begin{equation}\label{12}
A(z)=e^{zL_1}z^{-2L_0}g,
\end{equation}
where
\begin{equation}\label{113}
g(a)=e^{-\pi\sqrt{-1}(\tfrac{1}{2}p(a)+\D_a)}\phi(a),\quad a\in V,
\end{equation}
and $p(a)=0$ or $1$  stands for the parity of $a$ and $\D_a$ for its  $L_0$--eigenvalue. The definition of a $\phi$--invariant Hermitian form on a $V$--module
$M$  is similar (cf. Definition \ref{iM}).

The operator $A(z)$ with $g=(-1)^{L_0}$ appeared first in \cite{B} in the construction of the coadjoint module in the case when $V$ is purely even and the eigenvalues of $L_0$ are integers. Under the same assumptions on $V$ this operator was used in \cite{FLM} for the construction of the dual to the $V$--modules. 

Formula \eqref{12}  with $g=(-1)^{L_0}\phi$ was used in  \cite{DL}  to define unitary structures on vertex operator algebras and this notion was generalized in \cite{AiLin}
to  vertex algebras
with $\tfrac{1}{2}\Z_+$--grading compatible with parity, in which 
 case formula \eqref{113} simplifies to (see \eqref{gsuper})
$$g=(-1)^{L_0+2L_0^2}\phi.$$

As one can infer from the above remarks, the motivation for this definition stems from the observation that,  given a $V$--module $M$, one has (as in the Lie algebra case), a bijective correspondence between $\phi$--invariant Hermitian forms $(\cdot, \cdot)$ on $V$ and $V$--module conjugate linear
homomorphisms $\Theta:M\to M^\dagger$, where $M^\dagger$ is the  conjugate linear dual to $M$, with $V$--module structure defined by 
\begin{equation}\label{14}\langle Y_{M^\dagger}(a.z)m',m\rangle=\langle m', Y_M(A(z)a,z^{-1})m\rangle,\quad m\in M,m'\in M^\dagger.\end{equation}

Our first result, which generalizes  \cite[Theorem 5.2.1, Proposition 5.3.1]{FLM} (with a similar proof), is Proposition \ref{Contragredient}: formula \eqref{14} indeed defines a structure of a $V$--module on the restricted dual superspace $M^\dagger$ of $M$. Our second result, which generalizes, with the same proof,  that of \cite{Li}   in the symmetric case, is Proposition \ref{wt}, which describes $\phi$--invariant Hermitian forms on $V$. Its Corollary  \ref{hf} claims that a conformal vertex algebra $V$ with a  conjugate linear involution $\phi$ admits a (unique, up to a constant factor) 
 $\phi$--invariant Hermitian form if and only if any eigenvector of $L_0$ with eigenvalue $1$ is annihilated by $L_1$ (see also Remark \ref{V0dim1}). As usual, such a   Hermitian form can be expressed in terms of the expectation value on the vacuum (see formula \eqref{invform}).
 
 In Section \ref{eihf} we construct invariant Hermitian forms of fermionic, bosonic, affine and lattice vertex algebras. In  Section \ref{ihfm} we extend the results on invariant Hermitian forms on $V$ to arbitrary positive--energy (twisted) modules $M$. Proposition 5.3 claims that the space of $\phi$--invariant Hermitian forms on $M$ is isomorphic to the set 
 of $\omega$--invariant Hermitian forms on the module $M_0$ over the Zhu algebra. Here $M_0$ is the lowest energy subspace of $M$ and $\omega$ is the conjugate 
 linear anti--involution of the Zhu algebra, induced by the endomorphism of the superspace $V$ defined by 
 $$\omega(v)=A(1)v,\quad v\in V.$$
In Remark \ref{sr} we note that actually Proposition \ref{wt} is a special case of Proposition \ref{wtm}. 
 
In Section \ref{ihfw} we construct an invariant Hermitian form on the $W$--algebras  $W^k(\g,x,f)$ \cite{KRW}, \cite{KW1}. This construction is based on Proposition \ref{primary} (b), which says that the condition of Corollary \ref{hf}, that  all eigenvectors of $L_0$ with eigenvalue $1$ of the vertex algebra are  annihilated by $L_1$, holds, provided that the elements
 $h:=2x$ and $f$ can be included in a $sl(2)$-triple $\{e,f,h\}$. 
 
 In conclusion of this section we briefly discuss unitarity (i.e., positive semi-definiteness) of this Hermitian form for minimal $W$--algebras $W^k(\g,\theta/2)$. 
We show that the only interesting cases might occur when the $\tfrac{1}{2}\Z$--grading on the $W$--algebra is compatible with parity. In all the other cases we show that  the $W$--algebra  can be unitary only at collapsing levels \cite{AKMPP}, i.e. when the simple $W$-algebra $W_k(\g,\theta/2)$ ``collapses'' to its affine subalgebra: see Propositions \ref{7a}, \ref{7b}.
 These are just the first steps towards classification of unitary minimal $W$--algebras.
 
Throughout the paper the base field is $\C$. We also denote by $\Z_+$ the set of nonnegative integers and by $\nat$ the set of positive integers.

\section{Setup}\label{Section1}
\subsection{Basic definitions} Recall that a vector superspace  is a $\mathbb Z/2\Z$--graded vector space
 $V=V_{\bar{0}}\oplus V_{\bar{1}}$. The elements in $V_{\bar{0}}$ 
(resp. $V_{\bar{1}}$) are called even (resp. odd). Set
$$p(v)=\begin{cases}0\in\Z&\text{ if $v\in V_{\bar{0}}$},\\ 1 \in\Z&\text{ if $v\in V_{\bar{1}}$, }
\end{cases}
$$
i.e. we will regard  $p(v)$ as an integer, not as a residue class. We will often use the notation
\begin{equation}\label{s}\s(u)=(-1)^{p(u)}u,\qquad p(u,v)=(-1)^{p(u)p(v)}.\end{equation}

Let $V$ be a  vertex algebra.  We let 
 \begin{align}\label{0a3}
& Y:V \to (\mbox{End}\,V)[[z,z^{-1}]] ,\\
& v\mapsto Y(v,z)=\sum_{n\in{\Z}}v_{(n)}z^{-n-1}\ \ \ \  (v_{(n)}\in
\mbox{End}\,V),\notag
\end{align}
be the state--field correspondence. We denote by $\vac$ the vacuum vector in $V$ and by $T$ the translation operator (see e.g. \cite{KB} for details).

\begin{definition}\label{svoa}  In the present paper we will call a  vertex algebra $V$  {\it conformal}  if  there exists a distinguished vector $L\in
V_2$,  called a Virasoro vector, satisfying the following conditions:
\begin{align} \label{0a4}
& Y(L,z)=\sum_{n\in\Z}L_nz^{-n-2},\ [L_m,L_n]=(m-n)L_{m+n}+\frac{1}{12}(m^3-m)\delta_{m+n,0}c\,I ,\\
& L_{-1}=T,\\
& \text{$L_0$ is diagonalizable and its eigenspace decomposition has the form}\end{align}
\begin{equation}\label{g2.1}
V=\bigoplus_{n\in{ \frac{1}{2}\mathbb Z_+}}V_n,
\end{equation}
where 
\begin{equation}\label{dimf}\dim V_{n}< \infty \text{ for all $n$ and $V_0=\C\vac$}.
\end{equation}
The number $c$ is called the {\it central charge}.
\end{definition}
\begin{remark}\label{voas} 
Important examples of conformal vertex algebras are  vertex operator superalgebras, namely   the conformal vertex algebras  for which 
decomposition \eqref{g2.1} is compatible with parity, i.e.  $\sigma(u)=(-1)^{2L_0} u$.\par
In the definition of \cite{KB} of  conformal vertex algebras properties 
\eqref{g2.1} and \eqref{dimf} are not required. 
\end{remark}

By an automorphism of a conformal vertex algebra $V$ we mean a vertex algebra  automorphism  $\phi$ of $V$ (i. e. $\phi(u_{(n)}v)=\phi(u)_{(n)}\phi(v)$ for all $n\in\mathbb Z$)  with the  property that $\phi(L)=L$. Consequently, $\phi(V_n)=V_n$.
\par
The eigenvalues of $L_0$ on $V$ are called {\it conformal weights}; the conformal weight of $v\in V$ is denoted by $\D_v$, so that $v\in V_{\Delta_v}$.
The eigenvector $v$ of $L_0$ is called {\it quasiprimary} if $L_1v=0$ and {\it primary} if $L_nv=0$ for $n\ge 1$. One has for $v$ of conformal weight $\D_v$:
\begin{equation}
[L_\l v]=(L_{-1}+\D_v\l)v+\sum_{n\ge 2}\tfrac{\l^n}{n!}L_{n-1}v.
\end{equation}
Here and throughout  the paper we use the formalism of $\l$-brackets, which are defined by
$$
[u_\l v]= Res_z e^{z\l}Y(u,z)v,\quad u,v\in V.$$

Let $\Gamma$ be an  additive subgroup of $\R$ containing $\Z$.  If $\gamma\in \R$, denote by $[\gamma]$ its coset $\gamma+\Z$. 

\begin{definition}\label{def-gamma}Let $V$ be a conformal vertex algebra. A $\Gamma/\Z$--\emph{grading} on $V$
 is a map $\Upsilon:[\gamma]\mapsto V^{[\gamma]}\subseteq V$ such that $V$ decomposes as
 \begin{equation}
\label{eq:2.22}
  V\ =\ \bigoplus_{[\gamma] \in \Gamma /\Z}
V^{[\gamma]}
\end{equation}
and \eqref{eq:2.22} 
is a vertex algebra grading, compatible
with  $L_0$, i.e.
$$
V{}^{[\alpha]}{} _{(n)}V {}^{[\beta]}
\subseteq V{}^{[\alpha +\beta]} \ ,\ \ 
L_0(V{}^{[\gamma]}) \subseteq V{}^{[\gamma]}\ .
$$
\end{definition}

If $a\in V^{[\gamma]}$ then $[\gamma]$ is called the {\it degree} of $a$. Given a vector $a \in V$ of conformal weight $\Delta_a$ and 
degree $[\gamma]$, denote by $\epsilon_a$ the maximal
non--positive real number in the coset $[\gamma
-\Delta_a]$. This number has the following properties \cite{DK}:
\begin{equation}
  \label{eq:2.23}
  \epsilon_{\vac}=0 \, , \quad \epsilon_{Ta}=\epsilon_a \, , \quad
     \epsilon_{a_{(n)}b} = \epsilon_a + \epsilon_b +\chi (a,b)\, ,
\end{equation}
where $\chi (a,b)=1$ or $0$, depending on whether $\epsilon_a
+\epsilon_b \leq-1$ or not. 

Let $\gamma_a = \Delta_a
+\epsilon_a$ .  Then 
%
\begin{equation}
  \label{eq:2.24}
  \gamma_{\vac} = 0 \, , \quad \gamma_{Ta} = \gamma_a +1 \, , \quad
     \gamma_{a_{(n)}b}= \gamma_a + \gamma_b +\chi (a,b) -n-1\,.
\end{equation}
\subsection{Twisted modules}
\begin{definition}\label{modV}Let $\Gamma$ be an  additive subgroup of $\R$ containing $\Z$, and let $\Upsilon$ be a $\G/\Z$--grading on a conformal  vertex algebra $V$.
A $\Upsilon$--\emph{twisted module} for $V$ is a vector superspace $M$ and a parity preserving
linear map from $V$ to the space of $\End M$--valued 
$\Upsilon$--\emph{twisted quantum fields} 
$a \mapsto Y^M (a,z) = \sum_{m \in [\gamma_a]} a^M_{(m)}
z^{-m-1}$ (i.e. $a^M_{(m)} \in \End M$ and $a^M_{(m)} v=0$ for
each $v \in M$ and $m \gg 0$), such that the following properties hold:
\begin{gather}
  \vac _{(n)}^M = \delta_{n,-1} I_M \, ,\label{vacuumaxiom}\\
    \sum_{j \in \Z_+} \binom{m}{j} (a_{(n+j)}b)^M_{(m+k-j)}v\label{Borcherds}\\
    = \sum_{j \in \Z_+} (-1)^j \binom{n}{j} (a^M_{(m+n-j)}
       b^M_{(k+j)} -p(a,b)(-1)^n b^M_{(k+n-j)}a^M_{(m+j)}) v\, ,\notag
  \end{gather}
where $a \in V^{[\gamma_a]}$ ($\gamma_a\in \Gamma$), $m \in [\gamma_a]$, $n \in \Z$, $k \in [\gamma_b]$.
\end{definition}
The following Lemma is known; we prove it for completeness.
\begin{lemma} The Borcherds identity \eqref{Borcherds}  is equivalent to 
\begin{align}
&Res_{u}(i_{w,u}Y_M(Y(a,u)b,w)(w+u)^mu^nw^l)=\label{Resborcherds}\\
&Res_z(i_{z,w}Y_M(a,z)Y_M(b,w)z^m(z-w)^nw^l-p(a,b)i_{w,z}Y_M(b,w)Y_M(a,z)z^m(z-w)^nw^l)\notag
\end{align}
for all $n\in\Z$, $m\in[\gamma_a]$, $l\in[\gamma_b]$. As usual, $i_{x,y}$ means expanding in the domain $|x|>|y|$. 
\end{lemma}
\begin{proof}
Computing the residues  we find that \eqref{Resborcherds} is equivalent to
\begin{align*}
&\sum_{t\in\Z,j\in\Z_+}\binom{m}{j}(a_{(j+n)}b)^M_{(t-j+m+l)}w^{-t-1}\\
&=\sum_{t\in\Z,j\in\Z_+}(-1)^j\binom{n}{j}\left(a^M_{(m+n-j)}b^M_{(t+j+l)}-p(a,b)(-1)^{n}b^M_{(t+n-j+l)}a^M_{(m+j)})\right)w^{-t-1}\notag.
\end{align*}
\end{proof}

Since $V^{[\gamma]}$ is $L_0$--invariant, we have its eigenspace decomposition $V^{[\gamma]}=\oplus_\D V_\D^{[\gamma]}$, and we will write for $v\in V_{\Delta_v}^{[\gamma]}$,
$$
Y_M(v,z)=\sum_{n\in [\gamma-\D_v]}v^M_nz^{-n-\D_v}.
$$


\begin{definition}\label{pe}
A $\Upsilon$--twisted $V$--module $M$ is called a \emph{positive
  energy}  $V$--module if $M$ has an
$\R$--grading $M =\oplus_{{j \geq 0}} M_j$ such that
\begin{equation}
  \label{eq:2.39}
  a^M_n M_j \subseteq M_{j -n}, \ a\in V_{\D_a} .
\end{equation}
The subspace ~$M_0$ is called the
\emph{minimal energy subspace}.  Then,
\begin{equation}
  \label{eq:2.40}
  a^M_n M_0 = 0 \hbox{  for  } n>0 \hbox{  and  }a^M_0 M_0
  \subseteq M_0 \, .
\end{equation}
\end{definition}
\subsection{Zhu algebras}
 Set 
 \begin{equation}\label{vupsilon} V_{\Upsilon}=span(a\in V\mid \epsilon_a=0).\end{equation}
 Define a subspace $J_{\Upsilon}$ of $V$ as the span of elements
\begin{equation}\label{twistedZhu}
\sum_{j\in \Z_+}\binom{\gamma_a}{j}a_{(-2+\chi(a,b)+j)}b=Res_zz^{-2+\chi(a,b)}Y((1+z)^{\gamma_a}a,z)b,
\end{equation}
with $\epsilon_a+\epsilon_b\in\Z$.

Let
$$
a*b=\sum_{j\in \Z_+}\binom{\gamma_a}{j}a_{(-1+j)}b,
$$
Then $J_{\Upsilon}$ is a two sided ideal in $V_{\Upsilon}$ with respect to the product $*$. The quotient $Zhu_{\Upsilon}(V)=V_{\Upsilon}/J_{\Upsilon}$ is an associative superalgebra with respect to the product $*$ (see \cite{DK} for a proof), which is called the 
{\it Zhu algebra} associated to the grading \eqref{eq:2.22}.

\begin{example}
If $\Gamma$ is the subgroup of $\R$ spanned by the conformal weights $\D_a$ then one has a $\Gamma/\Z$--grading \eqref{eq:2.22}, for which 
$$
V^{[\gamma]}=\oplus_{\D\in[\gamma]}V_\D.
$$ 
The corresponding Zhu algebra is called the $L_0$--twisted (or Ramond twisted) Zhu algebra and denoted by $Zhu_{L_0} V$.
If $\Gamma=\Z$ then one has the trivial grading \eqref{eq:2.22}  by  setting $
V^\Z=V$.
The corresponding Zhu algebra is denoted by  $Zhu_{\mathbb Z} V$ and is called 
the non--twisted Zhu algebra (\cite{DK}, Examples 2.14 and 2.15).
\end{example}

\section{The conjugate contragredient module}\label{app}

In this section we adapt to our setting the proofs of Section 5 of \cite{FLM}, where the action of a vertex operator algebra on the linear dual of a module is defined. If $a\in V_{\D_a}$, set 
\begin{align}
(-1)^{L_0}a&=e^{\pi\sqrt{-}1\D_a}a,\quad\s^{1/2}(a)=e^{\frac{\pi}{2}\sqrt{-}1p(a)}a.
\end{align}
%
\begin{lemma}\label{twistAz}Let $g$ be a diagonalizable parity preserving conjugate  linear  operator on $V$ with modulus $1$ eigenvalues,  such that $g(L)=L$.
Then one has the relation
\begin{equation}\label{congsign}
g Y(a,z)g^{-1} b=p(a,b)Y(g (a),-z)b
\end{equation}
if and only if  the operator \begin{equation}\label{gamma}\phi=(-1)^{L_0} \s^{1/2} g\end{equation} is  a conjugate linear automorphism of $V$. Moreover
\begin{equation}\label{iff} g^2=I\iff \phi^2=I.\end{equation}
\end{lemma}
\begin{proof} Assume that $g$ satisfies \eqref{congsign}. Then 
\begin{align}\label{23}\phi(a)_{(n)}\phi(b)&=((-1)^{L_0} \s^{1/2}g)(a)_{(n)}((-1)^{L_0} \s^{1/2}g)(b)\\
&=e^{\pi\sqrt{-}1(\D_a+\D_b)}e^{\pi/2\sqrt{-}1(p(a)+p(b))}g(a)_{(n)}g(b).\notag
\end{align}
By \eqref{congsign}, $g(a_{(n)}b)=(-1)^{n+1} p(a,b)g(a)_{(n)}g(b)$. Substituting in \eqref{23}, and noting 
that $p(a)+p(b)+2p(a)p(b)=p(a_{(n)}b)\mod 4\Z$, we obtain
\begin{align*}\phi(a)_{(n)}\phi(b)&=e^{\pi\sqrt{-}1(\D_a+\D_b)}e^{\pi/2\sqrt{-}1(p(a)+p(b))}(-1)^{n+1} p(a,b)g(a_{(n)}b)\\
&=e^{\pi\sqrt{-}1\D_{a_{(n)}b}}e^{\pi/2\sqrt{-}1(p(a)+p(b)+2p(a)p(b))}g(a_{(n)}b)\\
&=e^{\pi\sqrt{-}1\D_{a_{(n)}b}}e^{\pi/2\sqrt{-}1p(a_{(n)}b)}g(a_{(n)}b)=\phi(a_{(n)}b).
\end{align*}
Reversing the argument we obtain the converse statement. 

To prove \eqref{iff} remark that  $g(L)=L$, hence $L_0g(a)=g(L)_0g(a)=g(L_0a)$, so, since $\D_a\in\R$,  $\D_{g(a)}=\D_a$. Moreover $g$   is parity preserving and conjugate linear, hence 
\begin{align*}\phi^2(a)&=(-1)^{L_0} \s^{1/2} g(-1)^{L_0} \s^{1/2} g(a)=e^{\pi\sqrt{-}1(\D_a+\tfrac{1}{2}p(a))}ge^{\pi\sqrt{-}1(\D_a+\tfrac{1}{2}p(a))}g(a)\\
&=e^{\pi\sqrt{-}1(\D_a+\tfrac{1}{2}p(a))}e^{-\pi\sqrt{-}1(\D_a+\tfrac{1}{2}p(a))}g^2(a)=g^2(a).
\end{align*}
\end{proof}

\begin{definition}\label{A} Let $g$ be a diagonalizable parity preserving conjugate linear  operator on $V$,  satisfying \eqref{congsign} and such that $g^2=I$.
Define $A(z):V\to V((z))$ by 
\begin{equation}\label{AZ}
A(z)v=e^{zL_1}  z^{-2L_0}gv,\quad v\in V.
\end{equation}
\end{definition}

\begin{lemma} 
We have
 \begin{equation}\label{conjAz}
p(a,b)A(w)Y(a,z)A(w)^{-1}b=i_{w,z}Y\left( A(z+w)a,\frac{-z}{(z+w)w}\right)b
\end{equation}
and
\begin{equation}\label{Azinverse}
A(z^{-1})=A(z)^{-1}.
\end{equation}
\end{lemma}
\begin{proof}
It is clear that 
\begin{equation}\label{conjzL0}
w^{-2L_0}Y(a,z)w^{2L_0}b=Y(w^{-2L_0}a,z/w^2)b.
\end{equation}
By \eqref{congsign}
\begin{equation}
p(a,b)g w^{-2L_0}Y(a,z)w^{2L_0}g^{-1}b=Y(g w^{-2L_0}a,-z/w^2)b.
\end{equation}
Finally we use that, if $|wz|<1$, then
\begin{equation}\label{conjL1}
e^{wL_1}Y(a,z)e^{-wL_1}=Y(e^{w(1-wz)L_1}(1-wz)^{-2L_0}a,\frac{z}{1-wz})
\end{equation}
(see (5.2.38) of \cite{FLM} and (4.9.17) of \cite{KB}) to get, for $|z|<|w|$,
\begin{equation}
p(a,b)e^{wL_1}Y(g w^{-2L_0}a,-z/w^2)e^{-wL_1}b=Y(e^{(w+z)L_1} g(w+z)^{-2L_0}a,\frac{-z}{w(w+z)})b,
\end{equation}
which is \eqref{conjAz}.

Since $g^2=I$, \eqref{Azinverse} is equivalent to
\begin{equation}\label{11}
A(z)a=g^{-1}z^{-2L_0}e^{-z^{-1}L_1}a=gz^{-2L_0}e^{-z^{-1}L_1}a.
\end{equation}
Next observe that
\begin{align*}
gz^{-2L_0}e^{-z^{-1}L_1}a&=\sum_rz^{-2L_0}(-1)^r\tfrac{1}{r!}g(L_1^ra)z^{-r}\\
&=\sum_rz^{-2\D_a}z^{2r}(-1)^r\tfrac{1}{r!}g(L_1^ra)z^{-r}\\
&=\sum_r(-1)^r\tfrac{1}{r!}g(L_1^ra)z^{r-2\D_a}.
\end{align*}
Since $g(L_1v)=-g(L)_1g(v)=-L_1g(v)$ we obtain
\begin{align*}
gz^{-2L_0}e^{-z^{-1}L_1}a=\sum_r\tfrac{1}{r!}L_1^rg(a)z^{r-2\D_a}=e^{zL_1}z^{-2L_0}g(a)=A(z)a.
\end{align*}
\end{proof}

\begin{remark}
Note that, by \eqref{11}, if $v$ is quasiprimary, we have
\begin{equation}\label{Az}
A(z)v=z^{-2\D_v} g(v).
\end{equation}
\end{remark}
If $\Upsilon$ is a $\Gamma/\Z$-grading on $V$, we let the \emph{opposite} grading $-\Upsilon$ be the grading defined by setting
$$
-\Upsilon([\gamma])=\Upsilon(-[\gamma]).
$$
We say that a $\Gamma/\Z$-grading is compatible with a map $\phi$ if $\phi(V^{[\gamma]})\subseteq V^{[\gamma]}$.

Let $M$ be a positive energy $\Upsilon$--twisted module and let $M^\dagger$ denote the restricted conjugate dual of $M$, that is 
\begin{equation}\label{degMdagger}
M^\dagger=\bigoplus_{n\ge 0} M^\dagger_n
\end{equation}
where $M^\dagger_n$ is the space of conjugate linear maps from $M_n$ to $\C$.
\begin{lemma}\label{eqrat} If $M+K\in\Z$, then
\begin{equation}\label{Rest}
Res_z z^M w^N i_{z,w} (z+w)^K =(-1)^{K+M-1} Res_z z^{-2 - K - M} w^{2 + 2 K + M + N} i_{w,z} (z+w)^{M}.
\end{equation}
\end{lemma}
\begin{proof} If $M+K<-1$, both sides of \eqref{Rest} are zero. If $M+K\ge -1$, then
\begin{align*}
Res_z z^M w^N i_{z,w} (z+w)^K&=Res_z \sum_{j\in \Z_+}\binom{K}{j} z^{M+K-j}w^{N+j}\\
&=\binom{K}{M+K+1}w^{N+M+K+1}.
\end{align*}
On the other hand
\begin{align*}
Res_z z^m w^n i_{w,z} (z+w)^k&=Res_z \sum_{j\in \Z_+}\binom{k}{j} z^{m+j}w^{n+k-j}\\
&=\binom{k}{-m-1}w^{n+m+k+1}\\
&=(-1)^{-m-1}\binom{-m-1-k-1}{-m-1}w^{n+m+k+1}.\\
\end{align*}
Equality holds for $m=-2 - K - M, n= 2 + 2 K + M + N, k = M.$
\end{proof}

\begin{theorem} \label{Contragredient}Let $\phi$ be a conjugate linear involution of a conformal  vertex algebra $V$. Choose $g$ as in Definition \ref{A} and define $A(z)$ by \eqref{AZ}. Let $\Upsilon$ be a $\Gamma/\Z$--grading on $V$ compatible with   $\phi$. Let $M$ be a $\Upsilon$--twisted positive energy module. Then 
\begin{enumerate}
\item[(a)] The  map $Y_{M^\dagger}$ given by
\begin{equation}\label{action}
\langle Y_{M^\dagger}(v,z)m', m\rangle=\langle m', Y_M(A(z)v, z^{-1})m\rangle,\ m\in M,m'\in M^\dagger,
\end{equation}
defines  on $M^\dagger$ the structure of a $(-\Upsilon)$--twisted $V$--module.
\item[(b)] If $\dim M_n<\infty$ for all $n$ then $(M^\dagger)^\dagger$ is naturally isomorphic to $M$.
\end{enumerate}
\end{theorem}
\begin{proof}Let $V=\oplus_{\gamma\in\Gamma/\Z} V^{\gamma}$ be the grading $\Upsilon$.
Write explicitly for $v\in V^{\gamma}_{\D_v}$,
$$
Y_{M^\dagger}(v,z)=\sum_{n\in -\gamma-\D_v}v^{M^\dagger}_{n}z^{-n-\D_v}.
$$
Then we have
$$\sum_n\langle v^{M^\dagger}_{n}m',m\rangle z^{-n-\D_v}=
\sum_n\langle m', \sum_t \tfrac{1}{t!}({L_1^t}g(v))^M_{n}m\rangle z^{n-\D_v}.
$$
In other words, if $n\in-\gamma-\D_v$, then
\begin{equation}\label{undagger}
\langle v^{M^\dagger}_{n}m',m\rangle =\langle m',\sum_t\tfrac{1}{t!}(L_1^tg(v))^M_{-n}m\rangle.
\end{equation}
In particular, $v^{M^\dagger}_nM^\dagger_j\subseteq M^\dagger_{j-n}$.
This proves that,  by \eqref{degMdagger}, $Y^{M^\dagger}$ is indeed a  $(-\Upsilon)$--twisted quantum field.

Next observe that
\begin{equation}\label{vacdagger}
\langle \vac^{M^\dagger}_{(n)}m',m\rangle=\langle \vac^{M^\dagger}_{n+1}m',m\rangle=\langle m',\vac^M_{-n-1}m\rangle=\d_{-n-1,0}\langle m',m\rangle,
\end{equation}
hence \eqref{vacuumaxiom} for $M^\dagger$ follows.

We now prove the Borcherds identity \eqref{Resborcherds} for $M^\dagger$, that is
\begin{align}
&Res_{u}\langle Y_{M^\dagger}(Y(a,u)b,w)i_{w,u}(w+u)^ku^nw^lm',m\rangle\notag\\
&=Res_z(\langle Y_{M^\dagger}(a,z)Y_{M^\dagger}(b,w)i_{z,w}z^k(z-w)^nw^lm',m\rangle)\\
&-p(a,b)Res_z(\langle Y_{M^\dagger}(b,w)Y_{M^\dagger}(a,z)i_{w,z}z^k(z-w)^nw^lm',m\rangle)\notag
\end{align}
for all $n\in\Z$, $k\in[-\gamma_a]$, $l\in[-\gamma_b]$.
Since
 \begin{align*}
\langle Y_{M^\dagger}(a,z)Y_{M^\dagger}(b,w)m',m\rangle&=\langle m',Y_M(A(w)b,w^{-1})Y_M(A(z)a,z^{-1})m\rangle,\\
\langle Y_{M^\dagger}(b,w)Y_{M^\dagger}(a,z)m',m\rangle&=\langle m',Y_M(A(z)a,z^{-1})Y_M(A(w)b,w^{-1})m\rangle,\\
\langle Y_{M^\dagger}(Y(a,u)b,w)m',m\rangle&=\langle m' ,Y_M(A(w)Y(a,u)b,w^{-1})m\rangle,
\end{align*}
we have to prove that
\begin{align*}
&Res_{u}(\langle m' ,Y_M(A(w)Y(a,u)b,w^{-1})m\rangle i_{w,u}(w+u)^ku^nw^l)\\
&=Res_z(\langle m',Y_M(A(w)b,w^{-1})Y_M(A(z)a,z^{-1})m\rangle i_{z,w}z^k(z-w)^nw^l)\\
&-p(a,b)Res_z(\langle m',Y_M(A(z)a,z^{-1})Y_M(A(w)b,w^{-1})m\rangle i_{w,z}z^k(z-w)^nw^l).
\end{align*}
Hence we need to check that
\begin{align}
&Res_{u}(Y_M(A(w)Y(a,u)b,w^{-1}) i_{w,u}(w+u)^ku^nw^l)\notag\\
&=Res_z(Y_M(A(w)b,w^{-1})Y_M(A(z)a,z^{-1})i_{z,w}z^k(z-w)^nw^l)\label{needtocheck}\\
&-p(a,b)Res_z(Y_M(A(z)a,z^{-1})Y_M(A(w)b,w^{-1}) i_{w,z}z^k(z-w)^nw^l).\notag
\end{align}
Changing variables in the Borcherds identity  \eqref{Resborcherds} for $Y_M$ we obtain, for all $n\in\Z$, $m\in[\gamma_a]$, $l\in[\gamma_b]$,
\begin{align*}
&Res_{t}Y_M(Y(a,t^{-1})b,w^{-1})i_{w^{-1},t^{-1}}(w^{-1}+t^{-1})^mt^{-n-2}w^{-l}\\
&=Res_t(Y_M(a,t^{-1})Y_M(b,w^{-1})i_{t^{-1},w^{-1}}t^{-m-2}(t^{-1}-w^{-1})^nw^{-l})\\
&-p(a,b)Res_t(Y_M(b,w^{-1})Y_M(a,t^{-1})i_{w^{-1},t^{-1}}t^{-m-2}(t^{-1}-w^{-1})^nw^{-l}),\notag
\end{align*}
which is equivalent to
\begin{align}
&Res_{t}(Y_M(Y(a,t^{-1})b,w^{-1})i_{t,w}(w+t)^{m}t^{-n-2-m}w^{-l-m}\notag\\
&=Res_{t}(Y_M(a,t^{-1})Y_M(b,w^{-1})i_{w,t}t^{-m-n-2}(w-t)^nw^{-l-n})\label{Jacres}\\
&-p(a,b)Res_{t}(Y_M(b,w^{-1})Y_M(a,t^{-1})i_{t,w}t^{-m-2-n}(w-t)^nw^{-l-n})\notag.
\end{align}
Write explicitly $A(w)a=\sum_{r\in\Z_+} C_r(a)w^{r-2\D_a}$, where $C_r(a)\in V$. Then
\begin{align*}
Y_M(A(t)a,t^{-1})&=\sum_{r\in\Z_+,h\in[\gamma_a]}C_r(a)_{(h)}t^{h+1}t^{r-2\D_a}=\sum_{r\in\Z_+}Y_M(C_r(a),t^{-1})t^{r-2\D_a},
\end{align*}
so, by \eqref{Jacres},
\begin{align*}
&Res_{t}(Y_M(A(t)a,t^{-1})Y_M(A(w)b,w^{-1})i_{w,t}t^{-m-n-2}(w-t)^nw^{-l-n})\\
&-p(a,b)Res_{t}(Y_M(A(w)b,w^{-1})Y_M(A(t)a,t^{-1})i_{t,w}t^{-m-2-n}(w-t)^nw^{-l-n})\\
&=\sum_{r}Res_{t}(Y_M(C_r(a),t^{-1})Y_M(A(w)b,w^{-1})i_{w,t}t^{-m-n-2+r-2\D_a}(w-t)^nw^{-l-n})\\
&-p(a,b)\sum_{r}Res_{t}(Y_M(A(w)b,w^{-1})Y_M(C_r(a),t^{-1})i_{t,w}t^{-m-2-n+r-2\D_a}(w-t)^nw^{-l-n})\\
&=\!\!\sum_{r}\!Res_{t}(Y_M(Y(C_r(a),t^{-1})A(w)b,w^{-1})i_{t,w}(w+t)^{m-r+2\D_a}t^{-n-2-m+r-2\D_a}w^{-l-m+r-2\D_a}\\
&=Res_{t}(Y_M(i_{t,w}Y(A(\frac{wt}{w+t})a,t^{-1})A(w)b,w^{-1})(w+t)^{m}t^{-n-2-m}w^{-l-m}.
\end{align*}
Therefore we have
\begin{align*}
&Res_{z}(Y_M(A(z)a,z^{-1})Y_M(A(w)b,w^{-1})i_{w,z}z^{k}(w-z)^{n}w^{l})\\
&-p(a,b)Res_{z}(Y_M(A(w)b,w^{-1})Y_M(A(z)a,z^{-1})i_{z,w}z^{k}(w-z)^{n}w^{l})\\
&=Res_{t}(Y_M(i_{t,w}Y(A(\frac{wt}{w+t})a,t^{-1})A(w)b,w^{-1})(w+t)^{-k-n-2}t^{k}w^{l+k+2n+2}.
\end{align*}
Hence
\eqref{needtocheck} turns into
 \begin{align*}
&Res_{t}(i_{t,w}Y_M(Y(A(\frac{wt}{w+t})a,t^{-1})A(w)b,w^{-1})(w+t)^{-k-n-2}t^{k}w^{l+k+2n+2}\\
&=-p(a,b)(-1)^nRes_{t}(Y_M(A(w)Y(a,t)b,w^{-1}) i_{w,t}(w+t)^kt^nw^l).
\end{align*}
Expand the L.H.S. above as
{\small
\begin{align*}
&Res_{t}(Y_M(i_{t,w}Y(A(\frac{wt}{w+t})a,t^{-1})A(w)b,w^{-1})(w+t)^{-k-n-2}t^{k}w^{l+k+2n+2})=\\
&\sum_{p,q,r,s}Res_t((C_r(a)_{(p)}C_s(b))_{(q)}i_{t,w}(w+t)^{-r+2\D_a-k-n-2}t^{k+r-2\D_a+p+1}w^{l+k+2n+2+q+s-2\D_b+r-2\D_a}),
\end{align*}
}
and apply Lemma \ref{eqrat} to obtain
\begin{align*}
&Res_{t}(Y_M(i_{t,w}Y(A(\frac{wt}{w+t})a,t^{-1})A(w)b,w^{-1})(w+t)^{-k-n-2}t^{k}w^{l+k+2n+2}\\
&=Res_{t}\sum_{p,q,r,s}(-1)^{p-n-1}(C_r(a)_{(p)}C_s(b))_{(q)}i_{w,t}(w+t)^{k+r-2\D_a+p+1}t^{n-p-1}w^{l+q+s+p+1-2\D_b}\\
&=(-1)^{n+1}Res_{t}\sum_{p}(-1)^{p}Y_M((A(w+t)a_{(p)}A(w)b),w^{-1})i_{w,t}(w+t)^{k+p+1}t^{n-p-1}w^{l+p}\\
&=(-1)^{n+1}Res_{t}(i_{w,t}Y_M(Y((A(w+t)a,\frac{-t}{w(w+t)})A(w)b),w^{-1})(w+t)^{k}t^{n}w^{l}).
\end{align*}
 Thus we are reduced to prove that
 \begin{align*}
&Res_{t}(i_{w,t}Y_M(Y(A(t+w)a,\frac{-t}{w(t+w)})A(w)b,w^{-1})w^{l}(t+w)^{k}t^{n}\\
&=p(a,b)Res_{t}(Y_M(A(w)Y(a,t)b,w^{-1}) i_{w,t}(w+t)^{k}t^nw^{l}),
\end{align*}
or equivalently
\begin{equation}
p(a,b)A(w)Y(a,t)b=i_{w,t}Y\left( A(t+w)a,\frac{-t}{(t+w)w}\right)A(w)b,
\end{equation}
which is equation  \eqref{conjAz} with   $A(w)b$ in place of $b$.  Claim (a) follows.


Let us now check (b). We need only to check that the map $m
\mapsto f_m\in (M^\dagger)^\dagger$ where $\langle f_m,m'\rangle=\overline{\langle m',m\rangle}$ is a $V$--module isomorphism. The map is clearly bijective since we are assuming $\dim M_n<\infty$. Now
\begin{align*}
\langle(Y_{(M^\dagger)^\dagger}(v,z)f_m,m'\rangle&=\langle f_m,Y_{M^\dagger}(A(z)v,z^{-1})m'\rangle=\overline{\langle Y_{M^\dagger}(A(z)v,z^{-1})m',m\rangle}\\
&=\overline{\langle m',Y_{M}(A(z)A(z^{-1})v),z)m\rangle}.
\end{align*}
Now use \eqref{Azinverse} to get
$$
\langle(Y_{(M^\dagger)^\dagger}(v,z)f_m,m'\rangle=\overline{\langle m',Y_{M}(v,z)m\rangle}=\langle f_{Y_M(v.z)m},m'\rangle.
$$
\end{proof}
 
\section{Invariant Hermitian forms on conformal vertex algebras} \label{S3} Let $V$ be a  conformal  vertex algebra.
By a Hermitian form on $V$ we mean a map $(\, \cdot \,\, , \, \cdot\, ): V\times V \to \mathbb C$ conjugate linear in the first argument and linear in the second, such that $(v_1,v_2)=\overline{(v_2,v_1)}$ for all $v_1,v_2\in V$.

Let $\phi$ be a conjugate linear  parity preserving involution of $V$. 
 Consider the conjugate linear operator  (cf \eqref{gamma})
\begin{equation}\label{g}g=((-1)^{L_0} \s^{1/2})^{-1}\phi.\end{equation}
By \eqref{iff}, we have that $g^2=I$.
Obviously $g$ satisfies the  hypothesis of Definition \ref{A}.
 Two instances of such a situation are the following.
\begin{enumerate}
\item Recall from  Remark \ref{voas} that, if $V$ is a vertex operator superalgebra, then  $(-1)^{2\D_a}=(-1)^{p(a)}$ for all $a\in V$. 
Set $s(a)=\D_a+\half p(a)$ and note that in this case $s(a)$ is an integer.   Then 
\begin{align*}
\D_a+2\D_a^2&=s(a)-\half p(a)+2(s(a)-\half p(a))^2\\&=s(a)-\half p(a)+2s(a)^2-2s(a)p(a)+\half p(a)^2.
\end{align*}
As $p(a)-p(a)^2=0$ and $p(a)$, $s(a)$ are integers, we see that
$$
\D_a+2\D_a^2\equiv s(a)\mod 2
$$
so that
$$
g(a)=e^{-\pi\sqrt{-}1(\D_a+\half p(a))}\phi(a)=(-1)^{s(a)}\phi(a)=(-1)^{\D_a+2\D_a^2}\phi(a)
$$
hence, if $V$ is a vertex operator superalgebra,
\begin{equation}\label{gsuper}
g=(-1)^{L_0+2L_0^2}\phi.
\end{equation}
\item
The vertex algebra of symplectic bosons provides an example of a conformal vertex algebra that is not a vertex operator superalgebra,  where our definition applies. Let $R_\R$ be a real  finite dimensional even vector space equipped with a bilinear non-degenerate symplectic form $\langle \cdot ,\cdot \rangle$. Let $R=\C\otimes R_\R$. Equip $R$ with the structure of a nonlinear conformal algebra with $\l$-bracket given by
$$
[a_\l b]=\langle a,b\rangle.
$$
Let $V$ be the corresponding universal enveloping vertex algebra.   The Virasoro vector is
$$
L=\half\sum :T(a^i)a_i:
$$
with $\{a_i\}$, $\{a^i\}$ dual bases of $R$. The elements in $R$ are primary of conformal weight $\tfrac{1}{2}$. Let  $\phi(r)=\bar r$, where  $\bar r$ is complex conjugation with respect to $R_\R$. Then, clearly,
$$
[\phi(a)_\l \phi(b)]=\overline{\langle a,b\rangle},
$$
hence $\phi$ extends to a conjugate linear involution  of $V$. In this case 
$$
g(r)=-\sqrt{-1}\bar r, \ r\in R.
$$
\end{enumerate}

\bigskip

The following definition first appeared in \cite{AiLin} for vertex operator superalgebras, generalizing the definition,  given in \cite{DL}, for vertex operator algebras.
\begin{definition}\label{invariant_form}
Let $\phi$ be a conjugate linear involution of a conformal  vertex algebra $V$. Choose $g$ as in Definition \ref{A} and define $A(z)$ by \eqref{AZ}. A Hermitian form $(\, \cdot \,\, , \, \cdot\, )$ on $V$ is said to be \emph{$\phi$--invariant} if, for all $a\in V$, 
\begin{equation}\label{i}
(v,Y(a, z)u)=(
Y(A(z)a, z^{-1})v,u),\quad u,v\in V.
\end{equation}
\end{definition}

\begin{remark}\label{quasiprimary} If $v\in V$ is quasi-primary, then, due to \eqref{Az},  \eqref{i} becomes
\begin{equation}\label{2.2} (v,a_nu)=(g(a)_{-n}(v),u),\quad u,v\in V.
\end{equation}
\end{remark}

The statement of the main result of \cite{Li} can be extended to our setting as follows.
\begin{theorem}\label{wt}In the setting of Definition \ref{invariant_form},
the space of  $\phi$--invariant Hermitian forms on $V $ is linearly isomorphic to the set of conjugate linear functionals 
$F\in V_0^\dagger$ such that $\langle F, L_1V_1\rangle=0$ and $\langle F, g(v)\rangle=\overline{\langle F, v\rangle}$ for all $v\in V_0$.
\end{theorem}
The proof is the same as in \cite[Theorem 3.1]{Li}. In the following we simply  check that the argument also works in our modified setting. Recall that an  element $m$ in a $V$--module $M$ is called vacuum--like if $a_{(n)}m=0$ for all $n\ge 0$ and all $a\in V$. 
By Proposition 2.3 of \cite{Li}, a vector $m\in M$ is vacuum-like if and only if $L_{-1}m=0$, i.e. the space of vacuum--like vectors is the space $M^{L_{-1}}$ of $L_{-1}$--invariants; moreover, if $m$ is a vacuum--like vector in $M$, then 
$Y(u,z)m=e^{zL_{-1}}u_{(-1)}m$.

Consider the map $$\Psi: Hom_V(V,M)\to M,\quad \Psi(\psi)=\psi(\vac).$$ By Proposition 3.4 of \cite{Li},  for any $V$--module $M$, $\Psi$ is an isomorphism between $\Hom_V(V,M)$ and the space $M^{L_{-1}}$.
\vskip10pt
\begin{proof}[Proof of Theorem \ref{wt}] 
Assume that $(\, \cdot \,\, , \, \cdot\, )$ is a $\phi$--invariant Hermitian form on $V$. Note that, since $g(L)=L$, \eqref{2.2} implies that $(L_0v,w)=(v,L_0w)$.  In particular  the eigenspaces of $L_0$ are orthogonal. Define $F\in V_0^\dagger$ by $\langle F, v\rangle=(v,\vac)$. Then $(\, \cdot \,\, , \, \cdot\, )$ is uniquely determined by $F$, since, letting $u=\vac$ and taking $Res_z z^{-1}$ of both sides of \eqref{i}, we obtain
$$
(v,a)=Res_z z^{-1}(Y(A(z)a,z^{-1})v,\vac).
$$
 By Remark \ref{quasiprimary}, 
\begin{equation}\label{2.2a}
(L_1v,\vac)=(v,L_{-1}\vac)=0,\ (\vac,L_1v)=(L_{-1}\vac,v)=0,
\end{equation}
hence, since $L_{-1}\vac=0$,
 we see that $\langle F, L_1V_1\rangle=0$. 
 
 Next we prove that, if $a\in V$, then
\begin{equation}\label{avac}
(g(a),\vac)=(\vac,a),
\end{equation}
Since the form is Hermitian, we have $(\vac,a)=\overline{(a,\vac)}$, 
so that \eqref{avac} implies $\langle F,g(a)\rangle=\overline{\langle F, a\rangle}$.
To prove \eqref{avac} we observe that, since $g(L)=L$, $g$ preserves the $L_0$--eigenspace decomposition. Since the eigenspaces of $L_0$ are orthogonal, we have that \eqref{avac} is satisfied if $\D_a\ne0$. We can therefore assume that $\D_a=0$, so that
\begin{align*}
(\vac,a)=Res_{z}z^{-1}(\vac,Y (a,z)\vac)&=Res_{z}z^{-1}(Y (A(z)a,z^{-1})\vac,\vac)\\
&=\sum_r(\tfrac{1}{r!}(L^r_1g(a))_{\D_a}\vac,\vac)\\
&=\sum_r\tfrac{1}{r!}((L^r_1g(a))_{0}\vac,\vac).
\end{align*}
By \eqref{2.2a}, in order to prove \eqref{avac}, we need only to prove that 
\begin{equation}\label{L1rv}
(L_1^rg(a))_{0}\vac\in L_1V_1,\ r\ge 1, a\in V_0.
\end{equation}
 We prove by induction on $r$ that 
$$
(L_1^rb)_{0}\vac\in L_1V_1,\ r\ge 1, b\in V_0.
$$
If $r=1$, then
$$
[L_1,b_{-1}]=\sum_{j\in\Z_+}\binom{2}{j}(L_{(j)}b)_0=(L_{-1}a)_0+2\D_bb_0+(L_1b)_0.
$$
Since $\D_b=0$, $(L_{-1}b)_0=0$, so
\begin{align*}
(L_1b)_0\vac=L_1(a_{-1}\vac)-b_{-1}(L_1\vac)=L_1(b_{-1}\vac)\in L_1V_1.
\end{align*}
If $r>1$, then
\begin{align*}
[L_1,(L_1^{r-1}b)_{-1}]&=\sum_{j\in\Z_+}\binom{2}{j}(L_{(j)}(L_1^{r-1}b))_0\\
&=L_{-1}(L_1^{r-1}b)_0-2(r-1)(L_1^{r-1}b)_0+(L_1^rb)_0\\
&=-(r-1)(L_1^{r-1}b)_0+(L_1^rb)_0\, ,
\end{align*}
so 
\begin{align*}
(L_1^rb)_0\vac&=L_1((L_1^{r-1}b)_{-1}\vac)-(L_1^{r-1}b)_{-1}L_1\vac+(r-1)(L_1^{r-1}b)_0\vac\\
&=L_1((L_1^{r-1}b)_{-1}\vac)+(r-1)(L_1^{r-1}b)_0\vac.
\end{align*}
The claim now follows by the induction hypothesis.

We now prove the converse statement. Consider $V$ as a $\Gamma/\Z$--graded vertex algebra with $\Gamma=\Z$ and the trivial grading $\Upsilon(\Z)=V$. Then the state--field correspondence defines on $V$ the structure of a $\Upsilon$--twisted positive energy module. Since $\Upsilon$ is clearly compatible with $\phi$, by Thoerem \ref{Contragredient}, we have a $\Upsilon$--twisted module structure on $V^\dagger$.
Fix $F\in V_0^\dagger$ which vanishes on $L_1V_1$. Then $F$ is a vacuum--like vector in $V^\dagger$. In particular the map  $\Phi_F:V\to V^\dagger$ defined by $\Phi_ F( v)= v^\dagger_{(-1)}F$ is a $V$--module homomorphism.
Here and in what follows we write for simplicity $a^\dagger_n$ instead of $a^{V^\dagger}_n$.
Define 
$$ 
(u,v)=\langle v^\dagger_{(-1)}F,u\rangle=\langle \Phi_F( v),u\rangle.
$$
Let us check that this form is $\phi$--invariant:
\begin{align*}
(v,Y(a,z)u)&=\langle \Phi_F (Y(a,z)u),v\rangle\\
&=\langle Y_{V^\dagger}(a,z) \Phi_F (u),v\rangle\\
&=\langle  \Phi_F (u),Y(A(z)a, z^{-1})v\rangle\\
&=(Y(A(z)a, z^{-1})v,u).
\end{align*}

It remains to show that, if $\langle F, a\rangle=\overline{\langle F,g(a)\rangle}$, then the form is Hermitian.
Since the form is $\phi$--invariant, by \eqref{avac},
$$
\overline{(a,\vac)}=\overline{\langle F,a\rangle}=\langle F,g(a)\rangle=(g(a),\vac)=(\vac,a).
$$
We can now check that the form is Hermitian:
\begin{align*}
\overline{(u,v)}&=Res_{z}z^{-1}\overline{(u,Y (v,z)\vac)}\\
&=\Res_{z}z^{-1}\overline{( Y (A(z)v,z^{-1})u,\vac)}\\
&=Res_{z}z^{-1}( \vac,Y (A(z)v,z^{-1})u)\\
&=Res_{z}z^{-1}(Y (A(z)A(z^{-1})v,z)\vac,u)\\
&=Res_{z}z^{-1}(Y (v,z)\vac,u)=(v,u),
\end{align*}
where, in the last step, we used \eqref{Azinverse}.\end{proof}

\begin{remark}\label{V0dim1}
Note that we didn't use in the proof  the assumptions that $V_0=\C\vac$ and that $\dim V_n=0$ for $n<0$. However, if $V_0=\C\vac$, then Theorem \ref{wt} implies that there exists a non--zero $\phi$--invariant Hermitian form on $V$ if and only if $V_1$ consists of quasiprimary elements, and for this form
$(\vac,\vac)\ne0$. The last statement follows observing that the eigenspaces of $L_0$ are orthogonal to each other and the kernel of a $\phi$--invariant Hermitian form is an ideal.
So if $(\vac,\vac)=0$, then $\vac$ lies in the kernel, hence the kernel of the form is $V$.  Also, such a Hermitian form, satisfying $(\vac,\vac)=1$, is unique.
\end{remark}
\begin{lemma}\label{sl2verma} Let $M$ be a module over $sl(2):=span\{e,h,f\}$, such that  $h$ is diagonalisable with finite--dimensional eigenspaces and negative eigenvalues.  Then $M$ is a direct sum of  Verma modules.
\end{lemma}
\begin{proof}Since the sum of $h$--eigenspaces with eigenvalues congruent mod 2 is a submodule, we may assume that all eigenvalues of $h$ are congruent mod $2$.
Since the $h$--eigenspaces are finite--dimensional, $U$ decomposes as the direct sum of the generalized eigenspaces for the Casimir operator  $\Omega$ of $sl(2)$. We can therefore assume that $\Omega$ has only one  eigenvalue.
Any irreducible subquotient of $M$ has negative highest weight, say $n$,
and the eigenvalue of $\Omega$ is $\half n^2 +n$ on it.
Hence if two irreducible subquotients with non--equal highest weights have the same $\Omega$--eigenvalue, the sum of these highest weights is $-2$.
Since all eigenvalues of $h$ are negative and congruent mod 2,
we deduce that all irreducible subquotients have the same highest weight $n$.
So, on the space $M^e$ of $e$--invariants (which is non-zero since the set of eigenvalues of $h$ is bounded above), $h$ has one eigenvalue $n$, and the same is true for any quotient of $M$. But on $N=M/U(sl(2))M^e$, $h$ has eigenvalues strictly smaller that $n$,
hence $N^e=0$ and $M=U(sl_2)M^e$. Since $n<0$, any vector from $M^e$ generates an irreducible Verma module, so $M$ is a direct sum of Verma modules with highest weight $n$.
 \end{proof}
\begin{proposition}\label{13} Let $V$ be a  conformal  vertex algebra  such that  $L_1V_1=0$, i.e. $V_{1}$ consists of quasiprimary 
vectors.
 Let $\{v_1, v_2,\ldots\}$ be a minimal system of strong generators, which 
includes $L$, and consists of eigenvectors for $L$. Then, summing to the $v_i$
elements from $L_{-1}V$, we can make these generators quasiprimary.
\end{proposition}
\begin{proof} 
By Lemma \ref{sl2verma}, applied to  $U=\bigoplus_{n>0} V_n$ and $f=L_{-1}, h=-2L_0, e=-\half L_1$, we get 
\begin{equation}\label{sl2dec}
V= \C\vac\oplus\sum_i M_i,
\end{equation}
 where the $M_i$  are Verma modules for $sl(2)$  with highest weight vectors 
quasiprimary 
elements. We proceed by induction on the conformal weight of a generator. If the conformal weight is $\half$ or  $1$ there is nothing to prove. Take now a generator $v_i$ whose  conformal weight 
is strictly greater than $1$. By \eqref{sl2dec}, we can write  $v_i=v'_i+L_{-1}b$, where $v'_i$  is quasiprimary and non-zero, due to minimality. 
By inductive assumption $b$ lies in the subalgebra  generated by 
quasiprimary generators. Hence we can replace $v_i$ by $v'_i$.
\end{proof}

Recall (cf. \cite{KB}) that, since $V_0=\C\vac$, one can define the expectation value $\langle v \rangle$ of $v$ by the equation $P_{V_0}(v)=\langle v\rangle \vac$ where $P_{V_0}$ is the projection onto $V_0$ with respect to the decomposition $V=V_0\oplus(\sum_{n\ne 0}V_n)$.
\begin{cor}\label{hf}
Suppose that $V$ is a conformal vertex algebra such that $V_1$ consists of 
quasiprimary vectors. Let $\phi$ be a conjugate linear involution of  $V$.
Then there exists a unique  $\phi$--invariant Hermitian form $(\, \cdot \,\, , \, \cdot\, )$ on $V$ such that $(\vac,\vac)=1$.
Moreover for any collection $\{U^i\mid i\in I\}$ of quasiprimary elements that strongly generate $V$ (it exists by Proposition \ref{13}) we have
\begin{align}
&\left((U^{i_1}_{j_1})^{m_1}\cdots (U^{i_t}_{j_t})^{m_t}\vac,(U^{i'_1}_{j'_1})^{m'_1}\cdots (U^{i'_{r}}_{j'_{r}})^{m'_{r}}\vac\right)\notag\\
&=\left\langle((g(U^{i_t})_{-j_t})^{m_t}\cdots(g(U^{i_1})_{-j_1})^{m_1}(U^{i'_1}_{j'_1})^{m_1'}\cdots (U^{i'_{r}}_{j'_{r}})^{m'_{r}}\vac\right\rangle.\label{invform}
\end{align}
\end{cor}
\begin{proof}
Since $L_1V_1=\{0\}$, the first statement follows from Theorem \ref{wt}.

To prove the second statement, note that, by \eqref{2.2}, for a quasiprimary element $U$, 
we have $(g(U)_nv,w)=(v,U_{-n}w)$ and
$$
(U_nv,w)=\overline{(w,U_nv)}=\overline{(g(U)_{-n}w,v)}=(v,g(U)_{-n}w).
$$
Since $V_0=\C\vac$, we have $(\vac,a)=\langle a\rangle$,
hence formula \eqref{invform} follows.
\end{proof}
\begin{definition}If the Hermitian form \eqref{invform} is positive definite, the vertex algebra $V$ is called {\it unitary}.
\end{definition}

\begin{lemma}\label{purely}
Let $V$ be a conformal vertex algebra and let $\phi$ be a conjugate linear involution on $V$. If there is a $\phi$--invariant  positive definite Hermitian form on $V$ and $a\in V$ is a non-zero quasiprimary element such that $\phi(a)=a$ then
\begin{align*}
&\langle a_{\D_a}a_{-\D_a}\vac\rangle\in \R\backslash\{0\}&&\text{if $(-1)^{2L_0}\s(a)=a$,}\\
&\langle a_{\D_a}a_{-\D_a}\vac\rangle\in \sqrt{-1} \R\backslash\{0\}&&\text{if $(-1)^{2L_0}\s(a)=-a$}.
\end{align*}
\end{lemma}
\begin{proof}
Since 
$$
(a,a)=e^{-\tfrac{\pi}{2}\sqrt{-1}(2\D_a+p(a))}\langle a_{\D_a}a_{-\D_a}\vac\rangle,
$$
and $(a,a)>0$, we see that  $\langle a_{\D_a}a_{-\D_a}\vac\rangle$ is real and non-zero if $(-1)^{2\D_a}\s(a)=a$, while it is purely imaginary and non-zero otherwise.
\end{proof}

In conclusion of this section we discuss invariant Hermitian forms on tensor products of vertex algebras. Recall from \cite{KB} that if $V, W$ are vertex algebras, their tensor product is the vertex algebra having   $V\otimes  W$ as space of states, $\vac\otimes 	\vac$ as vacuum vector and $T\otimes I + I \otimes T$ as translation operator. The state--field correspondence is given by 
$$Y(u \otimes v,z)=Y(u,z)\otimes Y(v,z).$$ 
If $V, W$ are conformal  vertex algebras, also $V\otimes  W$ is conformal: its Virasoro vector is $L=L_V\otimes \vac+\vac\otimes L_W$.

Let $\phi_V$, $\phi_W$ be conjugate linear involutions  of $V$, $W$ and set
$$g_V=((-1)^{(L_V)_0}\s_V^{1/2})^{-1}\phi_V,\quad g_{W}=((-1)^{(L_W)_0}\s_W^{1/2})^{-1}\phi_W, $$
$$g= g_V\otimes  g_W,\quad\phi=\phi_V\otimes \phi_W.
$$
Observe that  
$$
\phi=(-1)^{L_0}\s^{1/2}g.
$$
Moreover
$$
A(z)=e^{zL_1}z^{-2L_0}g=(e^{z(L_V)_1}\otimes e^{z(L_W)_1})(z^{-2(L_V)_0}\otimes z^{-2(L_W)_0})(g_V\otimes g_W)=A_V(z)\otimes A_W(z).
$$

If $(\, . \, , \, . \, )_V,\ (\, . \, , \, . \, )_W$ are invariant Hermitian forms on $V,W$, respectively, we can induce an invariant Hermitian form $(\, . \, , \, . \, )_{V\otimes W}$ on $V\otimes W$ by setting
$$
(v_1\otimes w_1 , v_2\otimes w_2 )_{V\otimes W}=(v_1,v_2)_V(w_1,w_2)_W.$$
Indeed, 
\begin{align*}
&(v_1\otimes v_2,Y(a\otimes b, z)(w_1\otimes w_2))=(v_1,Y(a,z)w_1)_V(v_2,Y( b, z)(w_2))_W\\
&=(Y(A_V(z)a,z^{-1})v_1,w_1)_V(Y( A_W(z)b, z^{-1})v_2,w_2)_W\\
&=(Y(A_V(z)\otimes A_W(z)(a\otimes b),z^{-1})(v_1\otimes v_2),w_1\otimes w_2)\\
&=(Y(A(z)(v_1\otimes v_2),w_1\otimes w_2).
\end{align*}
\section{Examples of invariant Hermitian forms}\label{eihf}
In this Section we apply Corollary \ref{hf} to fermionic, bosonic, affine, and lattice vertex algebras. 
\subsection{Superfermions}\label{ff}
Consider a superspace $A=A_{\bar 0}\oplus A_{\bar 1}$  endowed with a non-degenerate even skew-supersymmetric bilinear form
$(\, . \, | \, . \, )$.   Let $V(A)$ be the universal vertex algebra of the Lie  conformal superalgebra $A\oplus\C K$ with $\l$-bracket 
$$
[a_\lambda b]=(a|b) K,
$$
 $K$ being an even central element. Let $F$ be the fermionic vertex algebra: 
$$
F=V(A)/(K-\vac).
$$
Let  $\phi$ be a conjugate linear involution of $A$ such that
$$
(\phi(a)|\phi(b))=\overline{(a|b)}.
$$
 By setting $\phi(K)=K$ we can extend $\phi$ to a conjugate linear involution of $A\oplus \C K$. Indeed
$$
[\phi(a)_\l \phi(b)]=(\phi(a)|\phi(b))=\overline{(a|b)} K=\phi((a|b) K).
$$
This implies that $\phi$ extends to a conjugate linear involution of $V(A)$, hence, since $\phi(K-\vac)=K-\vac$, to an involution of $F$.

Fix a  basis $\{a^i\}$ of  $A$ and let $\{b^i\}$ be its dual basis w.r.t. $(\, . \, | \, . \, )$ (i.e. $(a^i|b^j)=\d_{i,j}$).
The Virasoro vector is \cite{KB}
\begin{equation}
  \label{eq:1.59}
L = \frac{1}{2} \sum_{i=1}^n : (T b^i)a^i:\,.
\end{equation}
It is easy to see that $\phi(L)=L$.
We embed $A$ in $F$ by identifying $v$ with  $:v\vac:$. It is easily checked that $v\in A$  is a primary element of $F$ of conformal weight $1/2$.
Set
\begin{equation}\label{gA}
g_A=((-1)^{L_0}\s^{1/2})^{-1}\phi.
\end{equation}
By \eqref{iff}, we have that $g_A^2=I$. Note that
\begin{equation}\label{phig}
g_A(a)=-\sqrt{-1} \phi(a),\ a\in A_{\bar0},\quad g_A(a)=-\phi( a), \ a\in A_{\bar1}.
\end{equation}

The set $\{a^i\}$  strongly and freely  generates  $F$. This means that, if we  order $(-\tfrac{1}{2}-\ganz_+)\times \{1,\dots,m+n\} $ lexicographically, then the set 
$$B=\bigcup_r\{(a^{i_1}_{j_1})^{h_1} \cdots (a_{j_r}^{i_r})^{h_r}\vac\mid (j_1,i_1)<\cdots <(j_r,i_r), h_s=1 \text{ if $p(a^{i_s})=1$}\}$$
 is a basis of $F$. 
With this choice one easily checks that 
$$F_0=\C\vac,\quad F_1=span_{\C}(\{:a^ia^j:\}).
$$
Since, by Wick formula \cite{KB},
\begin{align*}
[L_\l:a^ia^j:]&=:T(a^i)a^j:+ :a^iT(a^j): + \l:a^ia^j:+\int_0^\l
([T(a^i)_\mu a^j]+\half\l [a^i_\mu a^j] )d\mu\\
&=T(:a^ia^j:) + \l:a^ia^j:-\half \l^2 (a^i|a^j)+\half\l^2(a^i|a^j)\\
&=T(:a^ia^j:) + \l:a^ia^j:,
\end{align*}
we see that $L_1(F_1)=\{0\}$, hence Corollary \ref{hf} applies. Let $(\, \cdot \,\, , \, \cdot\, )$ be the unique invariant Hermitian form on $F$ such that $(\vac,\vac)=1$. By \eqref{phig} and \eqref{2.2}, the invariance amounts to
$$
(v_ja,b)=-\sqrt{-1}(a,\phi( v)_{-j}b),\ v\in A_{\bar0},\quad (v_ja,b)= -(a,\phi( v)_{-j}b),\ v\in A_{\bar1}$$
for all $a,b\in F$, $j\in\half+\ganz$.

We now discuss the unitarity of $F$. Assume that $F$ is unitary and $A_{\bar0}\ne 0$. Choose $a\ne0$ in $A_{\bar0}$; we can assume $\phi(a)=a$. Then, by Lemma \ref{purely}, $\langle a_{1/2}a_{-1/2}\vac\rangle\ne 0$ but $\langle a_{1/2}a_{-1/2}\vac\rangle=(a|a)=0$.
It follows that, if $F$ is unitary, then $A=A_{\bar1}$. Set $A_\R=\{a\in A\mid \phi(a)=-a\}$.
Then, if $a\in A_\R$,
$$
0<(a,a)=\langle a_{1/2}a_{-1/2}\vac\rangle=(a|a),
$$
so  $(\, . \, | \, . \, )_{A_\R\times A_\R}$  must be positive definite. In such a case,  choose $\{a^i\}$ to be an orthonormal basis of $A_\R$. 
It can be checked (say by induction on $r$) that
$$
\left\langle a^{i_t}_{-j_t}\cdots a^{i_1}_{-j_1}a^{i'_1}_{j'_1}\cdots a^{i'_{r}}_{j'_{r}}\vac\right\rangle=\d_{r,t}\prod_{s=1}^r\d_{i_s,i'_s}\prod_{s=1}^r\d_{j_s,j'_s}
$$
so the invariant Hermitian  form is the form defined by declaring the basis $B$ to be orthonormal. Hence $F$ is a  unitary conformal vertex algebra
 if and only 
if $A$ is purely odd.

\subsection{Superbosons}\label{bosons}
Let $\h$ be a  vector superspace equipped with a supersymmetric even bilinear form $(\, . \, | \, . \, )$. 
 Let $V(\h)$ be the universal vertex algebra of the Lie  conformal superalgebra $\h\oplus\C K$ with $\l$-bracket 
$$
[v_\lambda w]=\l(v|w) K,
$$
 $K$ being an even central element. Let $M(\h)$ be the vertex algebra: 
$$
M(\h)=V(\h)/(K-\vac).
$$

Let $\phi$ be a conjugate linear involution of $\h$. As in the previous example, if
$$
(\phi(a)|\phi(b))=\overline{(a|b)}.
$$
 we can extend $\phi$ to a conjugate linear involution of   $M(\h)$. 

 Fix a  basis $\{a^i\}$ of  $\h$ and let $\{b^i\}$ be its dual basis w.r.t. $(\, . \, | \, . \, )$ (i.e. $(a^i|b^j)=\d_{i,j}$).
The Virasoro vector is
 \begin{equation}
  \label{LSB}
L =  \tfrac{1}{2}\sum_{i=1}^n :  b^ia^i:\,.
\end{equation}
 It is easy to see that $\phi(L)=L$.
 
We embed $\h$ in $M(\h)$ by identifying $h$ with  $:h\vac:$. It is easily checked that $h\in\h$  is a primary element of $M(\h)$ of conformal weight $1$. 
 
Set
\begin{equation}\label{gAb}
g_\h=((-1)^{L_0}\s^{1/2})^{-1}\phi.
\end{equation}
By \eqref{iff}, we have that $g_\h^2=I$. Note that
\begin{equation}\label{phigh}
g_\h(a)=- \phi(a),\ a\in A_{\bar0},\quad g_\h(a)=\sqrt{-1}\phi( a), \ a\in A_{\bar1}.
\end{equation}

As in the previous example we can apply Corollary \ref{hf}, thus there is a unique  $\phi$--invariant Hermitian  form $(\, \cdot \,\, , \, \cdot\, )$ on $M(\h)$ such that $(\vac , \, \vac)=1$.

We now discuss the unitarity of $M(\h)$. Assume that $M(\h)$ is unitary and $\h_{\bar1}\ne 0$. Choose $h\ne0$ in $\h_{\bar1}$; we can assume $\phi(h)=h$. Then, by Lemma \ref{purely}, $\langle h_{1}h_{-1}\vac\rangle\ne 0$ but $\langle h_{1}h_{-1}\vac\rangle=(h|h)=0$.
It follows that, if $M(\h)$ is unitary, then $\h=\h_{\bar0}$. If this is the case,
set $\h_\R=\{h\in \h\mid \phi(h)=-h\}$. Then, as in Subsection \ref{ff}, we must have that $(\, . \, | \, . \, )_{\h_\R\times \h_\R}$ is positive definite.  We choose an orthonormal basis $\{a^i\}$ of $\h_\R$,
;  the $\phi$--invariant Hermitian form is therefore given by 
\begin{align*}
&\left((a^{i_1}_{j_1})^{m_1}\cdots (a^{i_t}_{j_t})^{m_t}\vac,(a^{i'_1}_{j'_1})^{m'_1}\cdots (a^{i'_{r}}_{j'_{r}})^{m'_{r}}\vac\right)\notag=\left\langle(a^{i_t}_{-j_t})^{m_t}\cdots(a^{i_1}_{-j_1})^{m_1}(a^{i'_1}_{j'_1})^{m_1'}\cdots (a^{i'_{r}}_{j'_{r}})^{m'_{r}}\vac\right\rangle.
\end{align*}
If we  order $(-\mathbb N)\times \{1,\dots,\dim \h\} $ lexicographically, then the set 
$$B=\bigcup_r\{a^{i_1}_{j_1} \cdots a_{j_r}^{i_r}\vac\mid (j_1,i_1)<\cdots <(j_r,i_r)\}$$
 is a basis of $M(\h)$. 
As in Example \ref{ff}, one can check that the basis $B$
is orthogonal; moreover the norm of each element is positive, so $M(\h)$ is a unitary vertex operator superalgebra, if and only if $\h$ is purely even.


\subsection{Affine vertex algebras}\label{Affine}Let $\fg $ be a simple Lie algebra or a basic classical simple
finite--dimensional Lie superalgebra and let $(\, . \, | \, . \, )$ be a   supersymmetric non-degenerate even invariant bilinear form on $\g$. 

We normalize the form $(\, . \, | \, . \, )$ on $\g$ by choosing an even highest root $\theta$ of $\g$ as in \cite{KW1} or \cite{AKMPP}, and  requiring $(\theta|\theta)=2$. 
If $\g=D(2,1,a)$, we assume $a\in\R$.

Let $\phi$ be a conjugate linear involution of $\g$. We assume that
$$
(\phi(x) |\phi(y))=\overline{(x|y)},
$$
noting that, if $\g$ is a Lie algebra, then the above assumption always holds.

Let 
$\Cur \g=\g\oplus\C K $
be the current Lie conformal algebra associated to $\g$  \cite{KB}. We extend $\phi$ to $\Cur\g$ by setting $\phi(K)=K$. 
Since 
$$
[\phi(x)_\l \phi(y)]=[\phi(x),\phi(y)]+\lambda(\phi(x)|\phi(y)K=\phi([x,y])+\lambda \overline{(x|y)}K=\phi([x_\l y]),$$
$\phi$ is a conjugate linear involution of $\Cur \g$, hence we can extend $\phi$ to a conjugate linear involution of the
 universal enveloping vertex algebra $V(\g)$ of $\Cur\g$.
 
 Choosing $k\in\R$, we note that $\phi(K-k\vac)=K-k\vac$, so $\phi$ pushes down to a conjugate linear involution of the the universal affine vertex algebra of level $k$.
 
We identify $a \in \fg$
with $:a\vac: \in V^k (\g)$. Let $h^{\vee}$  be the
dual Coxeter number of $\g$, i.e.
the eigenvalue of the Casimir operator $\sum_i b^ia^i$ on $\g$ divided by 2,  where $\{ a^i \}$ and $\{ b^i\}$ are  dual bases of
$\g$, i. e. $(a^i|b^j)=\delta_{ij}$.

A Virasoro vector is provided by the Sugawara construction (defined for $k\neq-h^{\vee}$), see e.g. \cite{KB}:
\begin{equation}
  \label{eq:1.57}
  L^{\fg} = \frac{1}{2(k+h^{\vee})} \sum_{i=1}^{\dim\g} : b^i a^i: \,.
\end{equation}
 It is easy to see that $\phi(L^\g)=L^\g$ provided that $k\in\R$.
 
 Set 
 $$g_\g=((-1)^{L_0}\s^{1/2})^{-1}\phi.
 $$
 Explicitly 
 $$
 g_\g(a)=-\phi(a),\ a\in\g_{\bar0},\quad  g_\g(a)=\sqrt{-1}\phi(a),\ a\in\g_{\bar1}.
 $$
It is well known that $a\in\g$ is a primary element of $V^k(\g)$ of conformal weight 1 (see e.g. \cite{KB}). Moreover,
the set $\{a^i\}$ strongly and freely generates $V^k(\g)$. It follows that
$
V^k(\g)_0=\C\vac$ and $L_1V^k(\g)_1=0$. By Corollary \ref{hf}, there exists a unique $\phi$--invariant Hermitian form on $V^k(\g)$, given by
\begin{align*}
&\left((a^{i_1}_{j_1})^{m_1}\cdots (a^{i_t}_{j_t})^{m_t}\vac,(a^{i'_1}_{j'_1})^{m'_1}\cdots (a^{i'_{r}}_{j'_{r}})^{m'_{r}}\vac\right)\notag\\
&=\left\langle(g_\g(a^{i_t})_{-j_t})^{m_t}\cdots(g_\g(a^{i_1})_{-j_1})^{m_1}(a^{i'_1}_{j'_1})^{m_1'}\cdots (a^{i'_{r}}_{j'_{r}})^{m'_{r}}\vac\right\rangle.\label{invform}
\end{align*}

If $k\ne-h^\vee$, the vertex algebra $V^k(\g)$ has a unique simple quotient that we denote by $V_k(\g)$. We now discuss the unitarity of $V_k(\g)$. Assume that there is a conjugate linear involution $\phi$ such that the corresponding $\phi$--invariant form on $V_k(\g)$ is positive definite.
If $\g$ is not a Lie algebra then there is $a\in \g_{\bar1}$,  $a\ne 0$. Since $\phi$ is parity preserving we can assume $\phi(a)=a$. Then
$$
(a,a)=(a_{-1}\vac,a_{-1}\vac)=\sqrt{-1}\langle a_{1}a_{-1}\vac\rangle=\sqrt{-1}k(a|a)=0.
$$
If $V_k(\g)$ is unitary, then $a$ is in the maximal ideal of $V^k(\g)$, hence $k=0$ and $V_k(\g)=\C$.

Assume now that $\g$ is a Lie algebra. Since $\phi$ is a conjugate linear involution of $V^k(\g)$ then $\phi_{|\g}$ is a conjugate linear involution of $\g$. Let $\g_\R$ be the corresponding real form.
As shown above, if $a\in\g_\R$, then 
$$
0<(a,a)=(a_{-1}\vac,a_{-1}\vac)=-\langle a_{1}a_{-1}\vac\rangle=-k(a|a),
$$
hence $(\, . \, | \, . \, )_{|\g_\R\times\g_\R}$ is either positive or negative definite.
Let $\overset{\circ}{\omega}_0$ be a compact conjugate linear involution  of $\g$ such that $\phi \overset{\circ}{\omega}_0=\overset{\circ}{\omega}_0\phi$. Let $\k_\R$ be the corresponding compact real form. Then
$$
\g_\R=\g_\R\cap \k_\R\oplus\g_\R\cap( \sqrt{-1} \k_\R).
$$
Since $(\, . \, | \, . \, )_{|\k_\R\times\k_\R}$ is negative definite and $\k_\R\cap\g_\R\ne\{0\}$, we see that $(\, . \, | \, . \, )_{|\g_\R\times\g_\R}$ is negative definite so $\phi=\overset{\circ}{\omega}_0$.
Let $\omega_0$ be the conjugate linear involution of the affinization  $\widehat\g$ of $\g$ which extends $\overset{\circ}{\omega}_0$ as in \S7.6 of \cite{VB}. Then the $\overset{\circ}{\omega}_0$--invariant Hermitian form on $V^k(\g)$ is defined by the property that
$$
(a_jx,y)=-(x,\overset{\circ}{\omega}_0(a)_{-j}y),\ a \in \g.
$$
It follows from Theorem 11.7 of \cite{VB} combined with the formula for $\omega_0$   given at page 103 of loc. cit., that the $\overset{\circ}{\omega}_0$--invariant Hermitian form on $V^k(\g)$ is positive semi--definite if and only if $k\in\ganz_+$.

\subsection{Lattice vertex algebras}\label{LVA} Let $Q$ be a positive definite integral lattice and $V_Q$ be its associated lattice vertex superalgebra (see e.g. \cite[\S 5.4]{KB}). Set $\h=\C\otimes_\Z Q$. Recall that the free bosons vertex operator algebra $M(\h)$ embeds 
in $V_Q=\oplus_{\a\in Q}(M(\h)\otimes \C e^\a)$ with parity $p(M(\h)\otimes e^\a)=(\a|\a)\mod 2$. Let $\{a^1,\ldots,a^l\}$ be an orthogonal basis of $\R\otimes_\Z Q$ and let $\{b^1,\ldots,b^l\}$ be the dual basis of $\h$  with respect to the form $(\, . \, | \, . \, )$ linearly extended from the form on $Q$. The Virasoro vector of $V_Q$ is
$$L=\frac{1}{2}\sum_{i=1}^l:a^ib^i:.$$ 
There are primary elements $e^\a\in V_Q$, $\a\in Q$ of conformal weight $\tfrac{1}{2}(\a|\a)$, such that a basis of  $V_Q$ is 
$$B=\bigcup_{r,\a}\{a^{i_1}_{j_1} \cdots a_{j_r}^{i_r}e^\a\mid (j_1,i_1)<\cdots <(j_r,i_r)\},$$
where, as in Example \ref{bosons}, $(-\mathbb N)\times \{1,\dots,\dim \h\} $ is ordered lexicographically.

Following \cite{DL}, we define a conjugate linear involution $\phi$ of $V_Q$ by setting 
\begin{equation}\label{phil}
\phi(a^{i_1}_{-j_1} \cdots a_{-j_r}^{i_r}e^\a)=(-1)^ra^{i_1}_{-j_1} \cdots a_{-j_r}^{i_r}e^{-\a}.
\end{equation}
It is immediate to see that $\phi(L)=L$. 
Since the conformal weight of $e^\a$ is $\half(\a|\a)$, we have that
$(-1)^{2L_0}\s=I$ so, if $g=((-1)^{L_0}\s^{1/2})^{-1}\phi$, then 
$$
g=(-1)^{L_0+2L_0^2}\phi.
$$
We have
$$
(V_Q)_0=\C\vac,\quad (V_Q)_1=span_\C(\{a^i\}\cup \{e^\a\mid (\a|\a)=2\}).
$$

Since the $a^i$, as well as the $e^\a$, are primary, we see that Corollary \ref{hf} applies. In particular the explicit expression for the $\phi$--invariant Hermitian form is
{\small \begin{align*}
&\left((a^{i_1}_{j_1})^{m_1}\cdots (a^{i_t}_{j_t})^{m_t}e^{\a},(a^{i'_1}_{j'_1})^{m'_1}\cdots (a^{i'_{r}}_{j'_{r}})^{m'_{r}}e^{\be}\right)\notag=
\delta_{\alpha,-\beta}\left\langle((a^{i_t}_{-j_t})^{m_t}\cdots(a^{i_1}_{-j_1})^{m_1}(a^{i'_1}_{j'_1})^{m_1'}\cdots (a^{i'_{r}}_{j'_{r}})^{m'_{r}}\vac\right\rangle.
\end{align*}}
As in Example \ref{bosons} one can check that the basis $B$
is orthogonal and consists of elements of positive norm, so $V_Q$ is unitary.

\section{Invariant Hermitian forms on modules}\label{ihfm}
Let $V$ be a conformal vertex algebra. 
Recall from \eqref{AZ} the definition of $A(z)$.
We let
\begin{equation}\label{omega}
\omega(v)= A(1)v, \ v\in V.
\end{equation}
Assume  that $V$ is $\Gamma/\Z$--graded and let $\Upsilon$ be a $\Gamma/\Z$--grading compatible with $\phi$. 

\begin{proposition}\label{J}  $\omega(J_{\Upsilon})\subseteq J_{-\Upsilon}$ so  $\omega$ induces a conjugate linear  anti--isomorphism of associative algebras
$\omega:Zhu_{\Upsilon}(V)\to Zhu_{-\Upsilon}(V)$.  
Moreover $\omega^2=Id$.
\end{proposition}
\begin{proof} By \eqref{twistedZhu}, we have
\begin{equation}\label{Jomega}
\omega\left(\sum_{j\in \Z_+}\binom{\gamma_a}{j}a_{(-2+\chi(a,b)+j)}b\right)=Res_w(w^{-2+\chi(a,b)}A(z)(Y((1+w)^{\gamma_a}a,w)b)_{|z=1}).
\end{equation}
By \eqref{congsign}
 \begin{align*}
&p(a,b)A(z)Y((1+w)^{\gamma_a}a,w)b=\\
&=p(a,b)e^{zL_1}z^{-2L_0}gY((1+w)^{\gamma_a}a,w)b\\
&=e^{zL_1}z^{-2L_0}Y((1+w)^{\gamma_a}g(a),-w)g(b).
\end{align*}
 By \eqref{conjzL0}
  \begin{align*}
&e^{zL_1}z^{-2L_0}Y((1+w)^{\gamma_a}g(a),w)g(b)=e^{zL_1}Y((1+w)^{\gamma_a}z^{-2L_0}g(a),-w/z^2)z^{-2L_0}g(b).
\end{align*}
By \eqref{conjL1}
\begin{align*}
&e^{zL_1}Y((1+w)^{\gamma_a}z^{-2L_0}g(a),-w/z^2)z^{-2L_0}g(b)\\
&=Y(e^{(z+w)L_1}(z+w)^{-2L_0}(1+w)^{\gamma_a}g(a),\tfrac{-w}{z(z+w)})e^{zL_1}z^{-2L_0}g(b),
\end{align*}
which means that
 \begin{align*}
&p(a,b)A(z)Y((1+w)^{\gamma_a}a,w)b=\\
&=Y(e^{(z+w)L_1}(z+w)^{-2L_0}(1+w)^{\gamma_a}g(a),\tfrac{-w}{z(z+w)})A(z)b,
\end{align*}
so that, since the grading is compatible with $\phi$ and $g(L)=L$,
 \begin{align*}
&(p(a,b)A(z)Y((1+w)^{\gamma_a}a,w)b)_{|z=1}=Y(e^{(1+w)L_1}(1+w)^{-L_0+\epsilon_a}g(a),\tfrac{-w}{(1+w)})\omega(b)\\
&=Y(e^{(1+w)L_1}(1+w)^{-L_0+\epsilon}g(a),\tfrac{-w}{(1+w)})\omega(b)
\end{align*}
Note that
$$
e^{(1+w)L_1}(1+w)^{-L_0+\epsilon}=(1+w)^{-L_0+\epsilon}e^{L_1}.
$$
Indeed, if $a\in V$,
\begin{align*}
&e^{(1+w)L_1}(1+w)^{-L_0+\epsilon}a=e^{(1+w)L_1}(1+w)^{-\D_a+\epsilon_{a}}a=(1+w)^{-\D_a+\epsilon_{a}}\sum_{r\ge 0} (1+w)^r\tfrac{1}{r!}L^r_1a
\\
&=\sum_{r\ge 0} (1+w)^{-\D_a+r+\epsilon_{a}}\tfrac{1}{r!}L^r_1a=(1+w)^{-L_0+\epsilon}\sum_{r\ge 0} \tfrac{1}{r!}L^r_1a\\&=(1+w)^{-L_0+\epsilon}e^{L_1}a.
\end{align*}
Hence,
 \begin{align}\label{calculationabove}
&(p(a,b)A(z)Y((1+w)^{L_0+\epsilon}a,w)b)_{|z=1}=Y((1+w)^{-L_0+\epsilon}e^{L_1}g(a),\tfrac{-w}{(1+w)})\omega(b)\\
&=Y((1+w)^{-L_0+\epsilon}\omega(a),\tfrac{-w}{(1+w)})\omega(b)=(1+w)^{-L_0}Y((1+w)^{\epsilon}\omega(a),-w)(1+w)^{L_0}\omega(b).\notag
\end{align}
Set 
$$
\varpi_a 
=\begin{cases}
-\epsilon_a-1&\text{if $\epsilon_a\ne0$},\\
0&\text{if $\epsilon_a=0$.}
\end{cases}
$$
Note that $\varpi$ is the function $\epsilon$ defined in Section \ref{Section1}  corresponding to the grading $-\Upsilon$.

Since $\epsilon_a+\epsilon_b\in\Z$, we have that $\varpi_a=-\chi(a,b)-\epsilon_a$ and $\chi(a,b)=1$ if and only if $\varpi_a+\varpi_b\le-1$. It follows that
 \begin{align*}
&Res_w(w^{-2+\chi(a,b)}(1+w)^{-L_0}Y((1+w)^{\epsilon}a,-w)(1+w)^{L_0}b)\\&=Res_w(w^{-2+\chi(a,b)}\sum_{n,j}(-1)^n\binom{-\D_{a}+\epsilon_a+n+1}{j}(a_{(n)}b)w^{-n-1+j})\\
&=\sum_{j}(-1)^j\binom{-\D_{a}+\epsilon_a+j+\chi(a,b)-1}{j}(a_{(-2+\chi(a,b)+j)}b)\\
&=\sum_{j}\binom{\D_{a}-\epsilon_a-\chi(a,b)}{j}(a_{(-2+\chi(a,b)+j)}b)\\
&=\sum_{j}\binom{\D_{a}+\varpi_a}{j}(a_{(-2+\chi(a,b)+j)}b)=Res_w(w^{-2+\chi(a,b)}Y((1+w)^{L_0+\varpi}a,w)b.
\end{align*}
Since $\Upsilon$ is compatible with $\phi$, we have that $\epsilon_{\omega(a)}=\epsilon_a$ (hence $\chi(a,b)=\chi(\omega(a),\omega(b)$). We find that 
\begin{align*}
&Res_w(w^{-2+\chi(a,b)}A(z)(Y((1+w)^{L_0}a,w)b)_{|z=1})\\
&=p(a,b)Res_w(w^{-2+\chi(\omega(a),\omega(b))}Y((1+w)^{L_0+\varpi}\omega(a),w)\omega(b)),
\end{align*}
hence, by \eqref{Jomega},
$
\omega(J_{\Upsilon})\subset J_{-\Upsilon}
$.

Next we prove that $\omega$ is an anti-automorphism. If $a\in V_{\Upsilon}$ (cf. \eqref{vupsilon}) then $\epsilon_a=\epsilon_{\omega(a)}=0$, thus, if $a,b\in V_{\Upsilon}$, by
\eqref{calculationabove},
 \begin{align*}
 &p(a,b) \omega(a*b)=p(a,b)Res_w(w^{-1}A(z)(Y((1+w)^{L_0}a,w)b)_{|z=1})\\
&=Res_ww^{-1}(1+w)^{-L_0}Y(\omega(a),-w)(1+w)^{L_0}\omega(b).
\end{align*}
 Now use  skew--symmetry
$
  Y (a,z)b = p(a,b) e^{zL_{-1}} Y(b,-z)a
$ (see e.g. \cite{KB}) 
to get
 \begin{align*}
 &\omega(a*b)\\
 &=Res_w(w^{-1}(1+w)^{-L_0}e^{-wL_{-1}}Y((1+w)^{L_0}\omega(b),w)\omega(a))\\
 &=Res_w\sum_{n,j,r}(-1)^r\binom{-\D_{\omega(a)}+n+1-r}{j}\tfrac{1}{r!}L_{-1}^r(\omega(b)_{(n)}\omega(a))w^{-n-2+j+r})\\
 &=\sum_{j,r}(-1)^r\binom{-\D_{\omega(a)}+j}{j}\tfrac{1}{r!}L_{-1}^r(\omega(b)_{(-1+j+r)}\omega(a))\\
 &=\sum_{r,j}(-1)^r\binom{-\D_{\omega(a)}+j}{j}\binom{-\D_{\omega(b)}-\D_{\omega(a)}+j+r}{r}(\omega(b)_{(-1+j+r)}\omega(a)\\
   &=\sum_{r,j}(-1)^{r+j}\binom{\D_{\omega(a)}-1}{j}\binom{-\D_{\omega(b)}-\D_{\omega(a)}+j+r}{r}(\omega(b)_{(-1+j+r)}\omega(a))
 \end{align*}
 \begin{align*}
  &=\sum_{r,j}(-1)^{r+j}\binom{\D_{\omega(a)}-1}{j}\binom{-\D_{\omega(b)}-\D_{\omega(a)}+j+r}{r}(\omega(b)_{(-1+j+r)}\omega(a))\\
 &=\sum_{n\ge r}(-1)^{n}\binom{\D_{\omega(a)}-1}{n-r}\binom{-\D_{\omega(b)}-\D_{\omega(a)}+n}{r}(\omega(b)_{(-1+n)}\omega(a))\\
&=\sum_{n}(-1)^{n}\binom{-\D_{\omega(b)}+n-1}{n}(\omega(b)_{(-1+n)}\omega(a)) \\
&=\sum_{n}\binom{\D_{\omega(b)}}{n}(\omega(b)_{(-1+n)}\omega(a))=\omega(b)*\omega(a).
\end{align*}
We used the fact that in $Zhu_{\Upsilon} V$ we have (cf. \cite[(2.35)]{DK})
 $$\frac{1}{r!}L_{-1}^ra=\binom{-\D_a}{r}a.$$
 and the Vandermonde identity on binomial coefficients.
 
Finally, by \eqref{Azinverse},
$$
\omega^2(a)=A(1)^2 a =a.
$$
hence $\omega^2=I$.
\end{proof}
\begin{remark} We now make explicit the map $\omega$ in the examples dealt with in Section \ref{S3}. In general, if $a$ is quasi-primary,
 we have, by \eqref{omega}
\begin{equation}\label{omegaqp}
\omega(a)=g(a).
\end{equation}
\begin{enumerate}
\item Let  $V=F$ be  the fermionic vertex algebra associated to a superspace $A$ as in Example \ref{ff}. According to \cite[Theorem 3.25]{DK},  $Zhu_{L_0}(V)$ is the Clifford algebra of $A$, i.e.  the unital associative algebra generated by $A$ with relations
$$[a,b]=(a|b),\quad a,b\in A.$$
 Then, according to \eqref{omegaqp} and \eqref{phig}, 
\begin{equation}
\omega(a)=-\sqrt{-1} \phi(a),\ a\in A_{\bar0},\quad \omega(a)=-\phi( a), \ a\in A_{\bar1}.
\end{equation}
\item  Let  $V=M(\h)$ be  the vertex algebra of superbosons associated to a superspace $\h$ as in Example \ref{bosons}. According to \cite[Theorem 3.25]{DK},  $Zhu_{L_0}(V)$ is the (super)symmetric  algebra of $A$.
 Then, according to \eqref{omegaqp} and \eqref{phigh}, 
\begin{equation*}
\omega(a)=- \phi(a),\ a\in A_{\bar0},\quad \omega(a)=\sqrt{-1}\phi( a), \ a\in A_{\bar1}.
\end{equation*}
\item If $V=V^k(\g)$ (cf. Example \ref{Affine}), then  $Zhu_{L_0}(V)=U(\g)$ (see e.g. \cite{DK}). Then, according to \eqref{omegaqp}, 
$$
\omega(a)=-\phi(a),\ a\in\g_{\bar0},\quad  \omega(a)=\sqrt{-1}\phi(a),\ a\in\g_{\bar1}.
 $$
\item If $V=V_Q$ is a lattice vertex algebra (cf. Example \ref{LVA}),  formulas \eqref{phil} and \eqref{omegaqp} give
$$\omega(e^\a)=(-1)^{\frac{(\a|\a)((\a|\a)+1)}{2}}e^{-\a},\quad \omega(h)=-\bar h,\,h\in\h.$$
Here $\bar h $ is the  conjugatie  of $h\in\h$ with respect to $\R\otimes_\Z Q$. 
If $Q$ is even,  $Zhu_{L_0}(V_Q)$ has been proved in \cite{DLM} to be isomorphic  to a generalized Smith algebra, denoted there by $\overline{A(Q)}$. The algebra $\overline{A(Q)}$  is generated by elements 
$E_\a,\,\a\in Q, h\in \h$, and the explicit formula for the isomorphism $Zhu_{L_0} V_Q\cong \overline{A(Q)}$ given in \cite[Theorem 3.4]{DLM} implies that 
$$\omega(E_\a)=(-1)^{\frac{(\a|\a)}{2}}E_{-\a},\quad \omega(h)=-\bar h,\,\ h\in\h,$$
is a conjugate linear anti--automorphism of $\overline{A(Q)}$.
\end{enumerate}
\end{remark}
\vskip20pt
\begin{definition} Let $R$ be an associative superalgebra over $\C$ with a conjugate linear anti--involution $\omega$, and let 
$M$ be an $R$--module. A  Hermitian form $(\,\cdot \,,\,\cdot\,)$ on $M$ is called $\omega$--invariant if 
$$
(\omega(a)m_1,m_2)=(m_1,a\,m_2),\ a\in R,\,m_1,m_2\in M.
$$
\end{definition}
Assume for the rest of this Section that $\Gamma=\Z$ or $\Gamma=\tfrac{1}{2}\Z$, so that $Zhu_{\Upsilon}=Zhu_{-\Upsilon}$. The following is the natural extension of Definition \ref{invariant_form} to $V$--modules. 
\begin{definition}\label{invariant_form_on_modules} Let $\phi$ be a conjugate linear involution of the vertex algebra $V$.
 A Hermitian form $(\, \cdot \,\, , \, \cdot\, )$ on a $\Upsilon$--twisted $V$--module $M$ is called  $\phi$--invariant if, for all $v\in V$, 
\begin{equation}\label{iM}
( m_1, Y_M(a, z)m_2)=( 
Y_M(A(z)a, z^{-1})m_1,m_2).
\end{equation}
\end{definition}
From now on we assume that the module $M$ is a positive energy module (see Definition \ref{pe}).
\begin{remark} The space of  $\phi$--invariant Hermitian forms on $M$ is linearly isomorphic to
$$\{\Theta\in Hom_V(M,M^\dagger)\mid \langle\Theta(m_1),m_2\rangle=\overline{\langle\Theta(m_2),m_1\rangle}\}$$
Indeed, given   $\Theta:M\to M^\dagger$  a $V$--module homomorphism, then  setting, for $m_1,m_2\in M$
$$
(m_1,m_2)_\Theta=\langle\Theta(m_2),m_1\rangle
$$
defines a $\phi$--invariant hermitian form on $M$. In fact
\begin{align*}
( m_1, Y_M(a, z)m_2)_{\Theta}&=\langle\Theta(Y_M(a, z)m_2),m_1\rangle=\langle Y_{M^\dagger}(a, z)\Theta(m_2),m_1\rangle\\
&=\langle \Theta(m_2), Y_M(A(z)v, z^{-1})m_1\rangle=(Y_M(A(z)v, z^{-1})m_1,m_2)_\Theta.
\end{align*}
Conversely, let $F:M\times M\to \C$ be a  $\phi$--invariant hermitian form; then $\Theta_F:M\to M^\dagger$ defined by 
$\langle\Theta_F(m_1),m_2\rangle=F(m_2,m_1)$ is a $V$--homomorphism from $M$ to $M^\dagger$. Indeed
\begin{align*}\langle\Theta_F(Y_M(a,z) m_1),m_2\rangle&=F(m_2,Y_M(a,z) m_1)=F(Y_M(A(z)a, z^{-1})m_2,m_1)\\&=\langle \Theta_F(m_1), Y_M(A(z)a, z^{-1})m_2\rangle=
\langle Y_{M^\dagger}(a, z)\Theta_F(m_1), m_2\rangle.
\end{align*}
\end{remark}

\vskip20pt

Recall that a positive energy $\Upsilon$--twisted $V$--module $M$ is said quasi--irreducible if it is generated by $M_0$ and there are no non-zero submodules $N\subset M$ such that $N\cap M_0=\{0\}$.

By \cite[Lemma 2.2]{DK}, if $M$ is a positive energy $\Upsilon$-twisted $V$--module, then the map $a\mapsto (a_0^M)_{|M_0}$ descends to define a $Zhu_{\Upsilon}V$--module structure on $M_0$.

\begin{lemma}\label{3.1}
If $M$ is quasi--irreducible then $M^\dagger$ is quasi--irreducible.
\end{lemma}
\begin{proof}
Set $N=VM_0^\dagger$. Then $N^\perp$ is graded and $\langle F, v\rangle=0$ for all $v\in N^\perp_0$, $F\in M_0^\dagger$. This implies that $N_0^\perp=\{0\}$, so $N^\perp=\{0\}$, hence $N=M^\dagger$. 

If $N$ is a graded submodule of $M^\dagger$ with $N_0=\{0\}$ then $N^\perp$ is a graded submodule of $M$ containing $M_0$. Since $M_0$ generates $M$, it follows that $N^\perp=M$ hence $N=\{0\}$.
\end{proof}
\begin{proposition}\label{wtm}
Let $M$ be a $\Upsilon$-twisted positive--energy $V$--module generated by $M_0$. Then
the space of  $\phi$--invariant Hermitian forms on $M$ is linearly isomorphic to the set of $\omega$--invariant Hermitian forms on
the $Zhu_{\Upsilon}V$--module $M_0$.
\end{proposition}

\begin{proof}If $(\, \cdot \,\, , \, \cdot\, )$ is a $\phi$--invariant Hermitian form on $M$, then $(\, \cdot \,\, , \, \cdot\, )_0=(\, \cdot \,\, , \, \cdot\, )_{|M_0\times M_0}$ is a $\omega$--invariant Hermitian form on $M_0$ by Proposition \ref{J}.

Let $(\, \cdot \,\, , \, \cdot\, )_0$ be a $\omega$--invariant Hermitian form on the $Zhu_{\Upsilon}V$--module $M_0$. 
Let $N$ be the sum of all graded submodules $N'$ of $M$ such that $N'\cap M_0=\{0\}$. Then $M/N$ is quasi--irreducible and $(M/N)_0=M_0$.
Define $\Phi_0:M_0\to M_0^\dagger$ by setting $\Phi_0(m_1)(m_2)=(m_2,m_1)_0$. Since the form $(\, \cdot \,\, , \, \cdot\, )_0$ is $\omega$--invariant, we have
\begin{align*}
\Phi_0(v^M_0m_1)(m_2)&=(m_2,v^M_0m_1)_0=(\omega(v)^M_0m_2,m_1)_0=\Phi_0(m_1)(\omega(v)^M_0m_2)\\&=(v^{M^\dagger}_0\Phi_0)(m_2)(m_1),
\end{align*}
so $\Phi_0$ is a $Zhu_{\Upsilon}(V)$--module map between $M_0$ and $M_0^\dagger$. By Lemma \ref{3.1} and \cite[Theorem 2.30]{DK}, there is a $V$--module map $\Phi:M/N\to (M/N)^\dagger$ such that $\Phi_{|M_0}=\Phi_0$. Define, for $m_1,m_2\in M$,
$$
(m_1,m_2)=\Phi(m_2+N)(m_1+N).
$$
It is clear that the form $(\, \cdot \,\, , \, \cdot\, )$ is $\phi$--invariant and that $(\, \cdot \,\, , \, \cdot\, )_0=(\, \cdot \,\, , \, \cdot\, )_{|M_0\times M_0}$. It remains to check that the form is Hermitian.

Consider the form $(\, \cdot \,\, , \, \cdot\, )'$ defined by $(m_1,m_2)'=\overline{(m_2,m_1)}$.
Note that $(\, \cdot \,\, , \, \cdot\, )'$ is $\phi$--invariant:
\begin{align*}
(m_1,Y_M(a,z)m_2)'&=\overline{(Y_M(a,z)m_2,m_1)}=\overline{(Y_M(A(z)A(z^{-1})a,z)m_2,m_1)}\\&=\overline{(m_2,Y_M(A(z)a,z^{-1})m_1)}=(Y_M(A(z)a,z^{-1})m_1,m_2)'.
\end{align*}
Since $(\, \cdot \,\, , \, \cdot\, )_0$ is Hermitian, then 
$$
(\, \cdot \,\, , \, \cdot\, )'_{|M_0\times M_0}=(\, \cdot \,\, , \, \cdot\, )_{|M_0\times M_0},
$$
hence $(\, \cdot \,\, , \, \cdot\, )'=(\, \cdot \,\, , \, \cdot\, )$.
\end{proof}

\begin{remark}\label{sr}
Theorem \ref{wt} is a consequence of Proposition \ref{wtm}. Indeed, the space of $\omega$--invariant Hermitian forms on $V_0$ is linearly isomorphic to $(V_0/L_1V_1)^\dagger$. The isomorphism is defined by mapping $(\, \cdot \,\, , \, \cdot\, )_0$ to $F_{(\, \cdot \,\, , \, \cdot\, )_0}$ where $F_{(\, \cdot \,\, , \, \cdot\, )_0}(v)=(v,\vac)_0$. To prove that this map is well defined, let us check that $F_{(\, \cdot \,\, , \, \cdot\, )_0}(L_1V_1)=0$. If $v\in V_1$, then
$$
L_1v=(L_1v)_0\vac=(v_0+(L_1v)_0)\vac)=\omega( g(v))_0\vac,
$$
so 
$$
F_{(\, \cdot \,\, , \, \cdot\, )_0}(L_1v)=(L_1v,\vac)_0=((\omega(g(v))_0\vac,\vac)_0=-(\vac, g(v)_0\vac)_0=0.
$$
The inverse is the map $F\mapsto (\, \cdot \,\, , \, \cdot\, )_F$, where $(v,w)_F=F(\omega(w)_0 v)$. Let us check that $(\, \cdot \,\, , \, \cdot\, )_F$ is $\omega$--invariant. If $u,v\in V_0$ and $w\in V_{\Z}$, then
$(u,w_0v)_F=F(\omega(w_0v)_0 u)$ and $(\omega(w)_0u,v)_F=F(\omega(v)_0 \omega(w)_0u)$.
Viewing $F$ as an element of $V^\dagger$, we observe that
$$
F(\omega(w_0v)_0 u)=((w_0v)^{V^\dagger}_0F)(u),\ F(\omega(v)_0 \omega(w)_0 u)=(w^{V^\dagger}_0v^{V^\dagger}_0F)(u),
$$
so it is enough to check that
\begin{equation}\label{daggermod}
(w_0v)^{V^\dagger}_0F=w^{V^\dagger}_0v^{V\dagger}_0F.
\end{equation}
Observe that, since $\langle F, L_1V_1\rangle=0$, $L_{-1}F=0$, $F$ is a vacuum--like element of $V^\dagger$. It follows from Proposition 3.4 of \cite{Li} that the map $\Phi:V\to V^\dagger$ defined by  $\Phi(a)= a^{V^\dagger}_{(-1)}F$ is a $V$--module map. In particular, 
$$
\Phi(a_{(n)}b)= a^{V^\dagger}_{(n)}\Phi(b)=a^{V^\dagger}_{(n)}(b^{V^\dagger}_{(-1)}F).
$$
On the other hand
$$
\Phi(a_{(n)}b)= (a_{(n)}b)^{V^\dagger}_{(-1)}F
$$
so 
$$
a^{V^\dagger}_{(n)}(b^{V^\dagger}_{(-1)}F)=(a_{(n)}b)^{V^\dagger}_{(-1)}F.
$$
Since $\D_v=\D_{w_0v}=0$, we find $v^{V^\dagger}_{(-1)}F=v^{V^\dagger}_0F$ and $(w_0v)^{V^\dagger}_{(-1)}F=(w_0v)^{V^\dagger}_{0}F$, so \eqref{daggermod} follows.
\end{remark}

\section{Invariant Hermitian forms on $W$--algebras}\label{ihfw}

We adopt the setting and notation of Section\ 1 of \cite{KW1}. We let  $W^k(\g,x,f)$ be the universal $W$--algebra of level $k\in \R$ associated to the datum  $(\fg ,x,f)$, where  $\fg$ is a
simple finite--dimensional Lie superalgebra with a reductive even part and a non-zero even
invariant supersymmetric bilinear form $(. \, | \, .)$, $x$ is an
$\ad$--diagonalizable element of $\fg$ with eigenvalues in
$\tfrac{1}{2}\ZZ$, $f$ is an even element of $\fg$ such that
$[x,f]=-f$ and the eigenvalues of $\ad x$ on the centralizer
$\fg^f$ of $f$ in $\fg$ are non-positive. Recall  that we are assuming that $a\in\R$ for $\g=D(2,1;a)$.  We call the datum $(\fg ,x,f)$ a \emph{Dynkin datum} if there is a $sl(2)$--triple $\{f,h,e\}$ containing $f$ and $x=\tfrac{1}{2} h$.

Let 
\begin{equation}\label{gg}\g=\bigoplus\limits_{j\in \frac{1}{2}\Z}\g_j
\end{equation}
 be the grading of $\g$ by $ad(x)$--eigenspaces.
We 
assume that $k\ne -h^\vee$ so that $W^k(\g,x,f)$ has a Virasoro vector.
Then
$W^k(\g,x,f)$ is 
a conformal vertex algebra in the sense of Definition \ref{svoa}. 
\begin{remark} It is easy to show that a datum  $(\g,x,f)$ as above is independent, up to isomorphism, from the choice of $f$, hence we may use notation $W^k(\g,x)$. 
\end{remark}

\begin{remark}\label{minimalW}
An important special case is when $f$ is a minimal nilpotent element 
of the even part of $\g$, i.e. $f$ is the root vector  $e_{-\theta}$ corresponding to a maximal even root $\theta$. In this case, the invariant bilinear form $(. \, | \, .)$ is normalized so that $(\theta | \theta)=2$. Choose the root vector $e_{\theta}\in\g_{\theta}$ in such a way that $(e_\theta|e_{-\theta})=\tfrac{1}{2}$. Setting
  $x=[e_\theta,e_{-\theta}]$, it is clear that  $(\g,x,e_{-\theta})$ is a Dynkin datum.  Identifying the Cartan subalgebra $\h$ with its dual using $(. \, | \, .)$, one has $x=\theta/2$.  The algebra 
$W^k(\g,\theta/2)$ is called a {\it minimal} $W$-algebra.
\end{remark}
\begin{lemma}\label{descend}
Let $\phi$ be a conjugate linear involution of $\g$ such that
\begin{equation}\label{defphi}
\phi(f)=f,\quad \phi(x)=x.
\end{equation}
Assume also, as in Subsection \ref{Affine}, that 
\begin{equation}\label{real}
\overline{(\phi(X)|\phi(Y))}=(X|Y),
\end{equation}
so that $\phi$ extends to a conjugate linear involution of $V^k(\g)$. Then $\phi$ descends to an involution of the vertex algebra $W^k(\g,x,f)$.
\end{lemma}
\begin{proof}
Let $A$ be the superspace $\Pi(\sum_{j>0}\g_j)$ where $\Pi$ is the reverse parity functor. Let $A^*$ be the linear dual of $A$ and set $A_{ch}=A\oplus A^*$. Define the form $\langle\,\cdot\,,\cdot\,\rangle_{ch}$ on $A_{ch}$ by setting, for $a,b\in A$, $a',b' \in A^*$,
$$
\langle a,b\rangle_{ch}=\langle a',b'\rangle_{ch}=0,\quad\langle a,b'\rangle_{ch}=b'(a),\quad \langle b',a\rangle_{ch}=-p(a,b')a'(b).
$$
Let $A_{ne}$ be the superspace $\g_{1/2}$ equipped with the form $\langle\,\cdot\,,\cdot\,\rangle_{ne}$ defined by
$$
\langle a,b\rangle_{ne}=(f|[a,b]).
$$
Since $\phi(f)=f$, 
$$
\langle\phi(a),\phi(b)\rangle_{ne}=(f|[\phi(a),\phi(b)])=(\phi(f)|\phi([a,b]))=\overline{(f|[a,b])}=\overline{\langle a,b\rangle}_{ne}.
$$
It follows that $\phi$ extends to a conjugate linear involution of $F(A_{ne})$. Similarly, setting $\phi(b^*)(a)=\overline{b^*(\phi(a))}$ for $b^*\in A^*$ and $a\in A$, we have
$$
\langle \phi(a),\phi(b^*)\rangle_{ch}=\phi(b^*)(\phi(a))=\overline{b^*(a)}=\overline{\langle a,b^*\rangle}_{ch},
$$
so $\phi$ extends to a conjugate linear involution of $F(A_{ch})$.
It follows that $\phi$ is a conjugate linear involution of the  vertex algebra $\mathcal C(\g,f,x,k)=V^k(\g)\otimes F(A_{ch})\otimes F(A_{ne})$.

Recall that there is an element $d\in\mathcal C(\g,f,x,k)$ such that $d_0$ is an odd derivation and $d_0^2=0$, making $\mathcal C(\g,f,x,k)$ a complex. It is easy to see that $\phi(d)=d$, hence the involution $\phi$ descends to an involution of the vertex algebra $W^k(\g,x,f)=H^0(\mathcal C(\g,f,x,k),d)$ \cite{KRW}, \cite{KW1}.
\end{proof}
Recall from \cite{KW1} that the vertex algebra $W^k(\g,x,f)$ is strongly and freely generated by fields $J^{\{x_i\}}$ with $\{x_i\}$ a basis of $\g^f$, the centralizer of $f$ in $\g$. We can clearly assume that the elements $x_i$ are homogeneous with respect to the gradation $\g^f=\oplus_j\g^f_j$. 
Let $\g_\R$ be the fixed point set of $\phi$. By \eqref{real}, we see that $(\, . \, | \, . \, )_{\g_\R\times\g_\R}$ is a real bilinear form. Since $\phi(x)=x$, we see that $\g_j=(\g_j\cap\g_\R)\oplus (\sqrt-1\g_j\cap\g_\R)$. Moreover $\langle\cdot,\cdot\rangle_{ne}$ is real when restricted to $\g_{1/2}\cap\g_\R$. Likewise, we can identify the real dual of $\g_+\cap\g_\R$ with the  set of $b^*\in A^*$ such that $\phi(b^*)=b^*$. It follows that we can identify the algebra $\mathcal C(\g_\R,f,x,k)$ as a real subalgebra of $\mathcal C(\g,f,x,k)$. We can therefore carry out the construction of the fields $J^{\{a\}}$ for $a\in\g_\R^f$ inside $\mathcal C(\g_\R,f,x,k)$ and therefore obtain that $\phi(J^{\{a\}})=J^{\{a\}}$. As $a\in \g^f$ can be written as $a=a_\R+ib_\R$ with $a_\R,b_\R\in \g_\R^f$, we see that we can construct the field  $J^{\{a\}}$ so that $\phi(J^{\{a\}})=J^{\{\phi(a)\}}$.

Let $L^\g$ the Virasoro vector for $V^k(\g)$ defined in \eqref{eq:1.57}. The vertex algebra $W^k(g,x,f)$ carries a Virasoro vector $L$, making it a conformal vertex algebra, which is  the homology class of  $
L^\g+T(x)+ L^{ch}+L^{ne}
$ (see \cite{KRW}).

In particular, by the above discussion and the explicit expressions for $L^\g$, $L^{ch}$, $L^{ne}$, we obtain that $\phi(L)=L$. Following   \eqref{gamma} we set
$$
g=((-1)^{L_0}\s^{1/2})^{-1}\phi.
$$

If $x_i\in\g^f_{j}$, then the conformal weight of $J^{\{x_i\}}$ is $1-j$. It follows that
$$
W^k(\g,x,f)_0=\C\vac,\quad W^k(\g,x,f)_1=span(\{J^{\{x_i\}}\mid x_i\in\g^f_0\}).
$$
\begin{theorem}\label{primary} (a) Let $v\in \g_0^f$. If $J^{\{v\}}\in W^k(\g,x,f)_1$ is quasiprimary for more than one  $k\in\C$, then
\begin{equation}\label{xv}
(x|v)=0.
\end{equation} 
 
 (b) If the datum  $(\fg,x,f)$ is a Dynkin datum, then the elements $J^{\{v\}}$ are primary for all $v\in \g_0^f$ and $k\in \C$ ($k\ne -h^\vee$). In particular, by Corollary \ref{hf}, there is a unique $\phi$--invariant Hermitian form 
$(\, \cdot \,\, , \, \cdot\, )$ on $W^k(\g,x,f)$ such that $(\vac,\vac)=1$. 

(c) Assume that $\g$ is a Lie algebra. If \eqref{xv} holds for  a datum $(\g,x,f)$ and all $v\in\g_0^f$, then it is a Dynkin datum.
\end{theorem}
\begin{proof} 
By  \cite[Theorem 2.4b]{KRW}, if $v\in\g_0^f$, then 
$$[L_\lambda J^{\{v\}}]=(T+\lambda)J^{\{v\}} +  \lambda^2 (\tfrac{1}{2}str_{\g_+} (ad\,v)-(k+h^\vee)(v|x)),$$
hence claim (a)  follows immediately.

If the datum $(\fg ,x,f)$ is a Dynkin datum, then $2(x|v)= ([e,f]|v)=(e|[f,v])=0$ if $v\in \g^f$. Hence for (b) it suffices to show
$str _{\g_j} (ad\,v)=0$ for all $j\in\half \mathbb N$ and $v\in \g_0^f$.

Consider the following bilinear form on $\g_j$:
$$<a,b>=((ad\,f)^{2j} a|b).$$
By $sl(2)$--representation theory, $(ad\,f)^{2j}:\g_j\to\g_{-j}$ is injective for $j>0$, hence  $<\cdot,\cdot>$ is  non-degenerate. The form is clearly
$ad\,\g_0^f$--invariant. 
The form is super (resp. skew--super) symmetric if $j\in\Z$ (resp. $j\in\tfrac{1}{2}+\Z$): 
$$<a,b>=((ad\,f)^{2j} a|b)=(-1)^{2j}(a|(ad\,f)^{2j}b)=(-1)^{2j}p(a,b)<b,a>.$$
Hence for  $v\in \g_0^f$,  $ad\, v$ lies in  $osp(\g_j)$ (resp. $spo(\g_j)$) if $j\in\Z$ (resp. $j\in\tfrac{1}{2}+\Z$). Hence in either case  its supertrace is $0$. This proves (b). 

By Theorem 1.1 from  \cite{EK}, $x=\tfrac{1}{2}h+c$, where $\{e,h,f\}$ is an $sl(2)$--triple for some $e\in \g_1$ and $c$ is a semisimple central element from the centralizer of this triple. We may assume that $c$ is defined over $\mathbb R$. But then $(x|c)=(\tfrac{1}{2}h+c|c)=(c|c)$. Since we are assuming 
that $\g$ is a simple Lie algebra,  \eqref{xv} implies that $c=0$, proving (c).
\end{proof}
\begin{remark} Let $\g$ be a simple Lie algebra. It follows from Theorem \ref{primary} that a datum $(\g,x,f)$ is Dynkin if and only if $(x|\g_0^f)=0 \,(\iff (x|\g^f)=0).$ In other words a $\tfrac{1}{2}\Z$--grading of $\g$ is Dynkin iff $f\in \g_{-1}$, all  eigenvalues of $ad x$ on $\g^f$ are non-positive and $(x|\g^f)=0$.
\end{remark}
\begin{example}
Let $\g=sl(3)$ with the data $(\g,\half(E_{11}-E_{33}),E_{31},k)$ and $(\g,-2E_{11}+ E_{22} +E_{33},E_{31}, k)$. The first one is a Dynkin datum corresponding to the minimal $W$--algebra $W^k(\g,\theta/2)$. The second one is not Dynkin: indeed, if $v=E_{11}-2E_{22}+E_{33}$, then $v\in\g_0^f$ and $(x|v)\ne0$.
\end{example}

\begin{cor}
Assume that $(\g,x,f)$ is a Dynkin datum. Then there is a unique $\phi$--invariant Hermitian form $(\, .\, ,\, .\,)$ on $W^k(\g,x,f)$ such that $(\vac,\vac)=1$.
\end{cor}
\begin{proof}
By Theorem \ref{primary} (b), we can apply Corollary \ref{hf}.
\end{proof}

We now describe the $\phi$--invariant Hermitian form more explicitly using formula \eqref{invform}.
 Fix a basis $\{x^i\}$ of $\g^f$.
 Set $\D_i=\D_{x^i}$ and $p_i=p(x^i)$. By Proposition \ref{13} we may assume that   the   fields $J^{\{x_i\}}$ are quasiprimary for all $i$. 
We can clearly assume  
that $\phi(x^i)=x^i$ for all $i$. Since  $\phi(L)=L$, the proof of Lemma \ref{sl2verma}, hence of Proposition \ref{13}, can be done over $\mathbb R$, so 
$\phi(J^{\{x^i\}})=J^{\{x^i\}}$ and let $J^{\{x^i\}}(z)=\sum\limits_{n\in -\D_i+\mathbb Z}J_n^{\{x^i\}}z^{-n-\D_i}$.
 
 Order the set
$$
\{(j,i)\in \tfrac{1}{2}\ganz_{+}\times \{0,\dots,\dim \g^f-1\} \mid j\in \D_i+\Z_+\}
$$
lexicographically. 
Then the set 
\begin{equation}\label{basisW}
\{(J^{\{x^{i_1}\}}_{-j_1})^{m_1}\cdots (J^{\{x^{i_t}\}}_{-j_t})^{m_t}\vac\mid m_i=0 \text{ or $1$ if $x^i$ is odd}\}
\end{equation}
is a basis of $W^k(\g,x,f)$. Since
$$
g(J^{\{ x^i\}})=(-\sqrt{-1})^{2\D_i+p_i}J^{\{ x^i\}},
$$  formula \eqref{invform} gives that
\begin{align}\label{formulona}
&\left((J^{\{x^{i_1}\}}_{j_1})^{m_1}\cdots (J^{\{x^{i_t}\}}_{j_t})^{m_t}\vac,(J^{\{x^{i'_1}\}}_{j'_1})^{m'_1}\cdots (J^{\{x^{i'_t}\}}_{j'_r})^{m'_t}\vac\right)\\
&=(-\sqrt{-1})^{\sum_rm_r(2\D_{i_r}+p_{i_r})}\left\langle (J^{\{x^{i_t}\}}_{-j_t})^{m_t}\cdots (J^{\{x^{i_1}\}}_{-j_1})^{m_1}(J^{\{x^{i'_1}\}}_{j'_1})^{m'_1}\cdots (J^{\{x^{i'_r}\}}_{j'_r})^{m'_r}\right\rangle.\notag
\end{align}

\begin{remark}\label{omegaW} Set $R=span(T^k(J^{\{x^i\}}), k\in\ganz^+)$. Let $\pi_Z$ be the quotient map from $W^k(\g,x,f)$ to $Zhu_{L_0}(W^k(\g,x,f))$.  Set $\mathfrak w=span(\pi_Z(J^{\{x^i\}}))$. By \eqref{basisW} the set
$$
\{:(T^{k_1}J^{\{x^{i_1}\}})^{m_1}\cdots (T^{k_t}J^{\{x^{i_t}\}})^{m_t}:\mid m_i=0 \text{ or $1$ if $x^{j_i}$ is odd}\}
$$
is a basis of $W^k(\g,x,f)$.
It follows from  Theorem 3.25 of \cite{DK} that 
$$
R/(L_{-1}+L_0)R\simeq \mathfrak w
 $$
 has the structure of a nonlinear Lie superalgebra and that $Zhu_{L_{0}}(W^k(\g,x,f))$ is its universal enveloping algebra. In particular the set 
 $$
 \{(\pi_ZJ^{\{x^{i_1}\}})^{m_1}*\cdots *(\pi_ZJ^{\{x^{i_t}\}})^{m_t}\mid m_i=0 \text{ or $1$ if $x^{j_i}$ is odd}\}
 $$
 is a basis of $Zhu_{L_{0}}(W^k(\g,x,f))$. Since, by Proposition \ref{13},  $J^{\{x_i\}}$ can be chosen to be quasiprimary for all $i$, it is clear that the involution $\omega$ in this basis is given by
 $$
 \omega((\pi_ZJ^{\{x^{i_1}\}})^{m_1}*\cdots *(\pi_ZJ^{\{x^{i_t}\}})^{m_t})=(-\sqrt{-1})^{\sum_rm_r(2\D_{i_r}+p_{i_r})}(\pi_ZJ^{\{x^{i_t}\}})^{m_t}*\cdots *(\pi_ZJ^{\{x^{i_1}\}})^{m_1}.
 $$
\end{remark}

We now restrict to  the case of a minimal $W$--algebra $W^k(\g,\theta/2)$ (see Remark \ref{minimalW}) where one has a more explicit description of $Zhu_{L_0}(W^k(\g,\theta/2))$ and its involution.

Set $\g^\natural=\g_0^f$. Then $\g^f=\g^\natural\oplus\g_{-1/2}\oplus\C f$. The elements $J^{\{v\}}$  are uniquely determined for $v\in \g^\natural\oplus \g_{-1/2}$ and have been computed explicitly in \cite{KRW}. One usually denotes $J^{\{v\}}$ by $G^{\{v\}}$ if $v\in \g_{-1/2}$. We also write $\g^\natural=\oplus_{i=0}^r\g_i$ with $\g_0$ the (possibly zero) center and $\g_i$ a simple ideal for $i>0$.

Set, for $u,v\in \g_{-1/2}$,
$$
\langle u,v\rangle=(e_{\theta}|[u,v])
$$
and note that $\langle\,\cdot\,,\,\cdot\,\rangle$ is a $\g^\natural$--invariant skew--supersymmetric bilinear form on $\g_{-1/2}$.
Fix a basis $\{a_i\}$ of $\g^\natural$ and a basis $\{u_i\}$ of $\g_{-1/2}$. Then  $W^k(\g,\theta)$ has as set of free generators 
$$
\{J^{\{a_i\}}\}\cup\{G^{\{u_i\}}\}\cup\{L\}.
$$
Moreover the $\l$--brackets between generators is known explicitly 	\cite{KRW}, \cite{KW1}, \cite{AKMPP}, \cite{KMP}, and Section \ref{6}: $L$ is the Virasoro vector and its central charge is 
$\frac{k\,\sdim\g}{k+h^\vee}-6k+h^\vee-4$, the  $J^{\{u\}}$ are primary of conformal weight $1$, the $G^{\{v\}}$ are primary of conformal weight $\frac{3}{2}$ and
\begin{enumerate}
\item $[{J^{\{a\}}}_\l J^{\{b\}}]=J^{\{[a,b]\}}+\l\d_{ij}(k+\frac{h^\vee- h^\vee_{0,i}}{2})(a|b)$ for $a\in \g^\natural_i$, $b\in \g^\natural_j$;
\item $[{J^{\{a\}}}_\l G^{\{u\}}]=G^{\{[a,u]\}}$ for  $ u\in \g_{-1/2}$, $a\in \g^\natural$;
\item
\begin{align*}
&[{G^{\{u\}}}_{\lambda}G^{\{v\}}]=-2(k+h^\vee)\langle u,v\rangle L+\langle u,v\rangle\sum_{\alpha=1}^{\dim \g^\natural} 
:J^{\{a^\alpha\}}J^{\{a_\alpha\}}:+\\\notag
&2\sum_{\a,\be=1}^{\dim\g^\natural}\langle[a_\a,u],[v,a^\be]\rangle:J^{\{a^\a\}}J^{\{a_\be\}}:
 +2(k+1) (\partial+ 2\lambda) J^{\{[[e_{\theta},u],v]^{\natural}\}}\\\notag
&+ 2 \lambda \sum_{\a,\be=1}^{\dim\g^\natural}\langle[a_\a,u],[v,a^\be]\rangle
  J^{\{ [a^\a,a_\be]\}}
        + 2 p(k)\l^2\langle u,v\rangle.
\end{align*}.
\end{enumerate}

Here $\{a_\alpha\}$ (resp. $\{u_\gamma\}$) is a basis of  $\g^\natural$ (resp. $\g_{1/2}$) and $\{a^\alpha\}$ (resp. $\{u^\gamma\}$) is the corresponding dual basis w.r.t. $(\, . \, | \, . \, )$ (resp w.r.t. $\langle\cdot,\cdot\rangle_{\rm ne}=(e_{-\theta}|[\cdot,\cdot])$),  $a^\natural$ is the  orthogonal projection of $a\in\g_{0}$ on $\g^\natural$, $a_i^\natural$ is the projection of $a^\natural$ on the $i$th minimal ideal $\g_i^\natural$ of $\g^\natural$, $k_i=k+\frac{1}{2}(h^\vee-h^\vee_{0,i})$, where $h^\vee_{0,i}$ is the dual Coxeter number of $\g_i^\natural$ with respect to the restriction of the form $(\, . \, | \, . \, )$, and  $p(k)$ is the monic quadratic polynomial given in Table 4 of \cite{AKMPP}. See Appendix \ref{6} for the derivation of formula (3) from the formulas given in \cite{KRW}.

Identify $ \mathfrak w$ with $\g^\natural\oplus \g_{-1/2}\oplus \C L$ by identifying $\pi_ZJ^{\{a\}}$ with $a$, $\pi_ZG^{\{v\}}$ with $v$ and $\pi_zL$ with $L$.
As in Remark \ref{omegaW}, a basis of $Zhu_{L_0}(W^k(\g,\theta))$ is given by
 $$
 \{u_{i_1}^{m_1}*\cdots *u_{i_t}^{m_t}*a_{j_1}^{n_1}*\cdots*a_{j_r}^{n_r}*L^k\!\mid\! i_1<\cdots i_t;\,j_1<\cdots < j_r;\, m_p,n_q\in\{0,1\} \text{ if $a_{i_p}$ or $u_{j_q}$ is odd}\}.
 $$ 
Moreover the commutation relations among the generators are as follows (here $[\cdot,\cdot]_\g$ denotes the bracket in $\g$, while $[\cdot,\cdot]$ is the bracket in $Zhu_{L_0}(W^k(\g,\theta)$).
\begin{enumerate}
\item $L$ is a central element,
\item $[a,b]=[a,b]_\g$ if $a,b\in\g^\natural$,
\item $[a,v]=[a,v]_\g$ if $a\in \g^\natural$ and $v\in\g_{-1/2}$,
\item  \begin{align*}
[u,v]=&\langle u,v\rangle\left(\sum_{\a=1}^{\dim\g^\natural} (a^\a*a_\a-[a^\a,a_\a]_\g)-2(k+h^\vee) L-\half p(k)\right)\\&+\sum_{\a,\be=1}^{\dim\g^\natural}\langle[a_\a,u]_\g,[v,a^\be]_\g\rangle (2a^\a*a_\be-[a^\a,a_\be]_\g).
\end{align*}
\end{enumerate}
By (2), (3) we can drop the subscript $\g$ from the bracket. Moreover observe that 
$$
\sum_{\a=1}^{\dim\g^\natural} [a^\a,a_\a]_\g=0
$$
and that
$$
2a^\a*a_\be-[a^\a,a_\be]_\g=2a^\a*a_\be-[a^\a,a_\be]=a^\a*a_\be+p(a_\a,a_\be)a_\be*a^\a.
$$
Setting $L'=2(k+h^\vee) L+\half p(k)$, a new  generating space is
$
\g^\natural\oplus \g_{-1/2}\oplus \C L'
$
and the commutation relations are (1) with $L'$ in place of $L$, (2), (3) and
\begin{enumerate}
\item[(4')]  \begin{align*}
[u,v]=&\langle u,v\rangle\left(\sum_{\a=1}^{\dim\g^\natural} a^\a*a_\a-L'\right)+\sum_{\a,\be=1}^{\dim\g^\natural}\langle[a_\a,u]_\g,[v,a^\be]_\g\rangle (a^\a*a_\be+p(a_\a,a_\be)a_\be*a^\a).
\end{align*}
\end{enumerate}
It is then clear that $Zhu_{L_0}(W^k(\g,\theta/2))$ does not depend on $k$ if $k\ne -h^\vee$.

The involution $\omega$ is easily computed: since the generators are quasiprimary, we have by \eqref{omegaqp}: $\omega(J^{\{a\}})=g(J^{\{a\}})$, hence 
\begin{align*}
&\omega(L')=L',\\
&\omega(a)=(-1)^{p(a)+1}(\sqrt{-1})^{p(a)}\phi(a),\ a\in\g^\natural,\\
&\omega(v)=(-1)^{p(v)}(\sqrt{-1})^{p(v)+1}\phi(v),\ v\in\g_{-1/2}.
\end{align*}

Recall that, if $k+h^\vee\ne0$, then $W^k(\g,\theta/2)$ has a unique simple quotient $W_k(\g,\theta/2)$. Remark  that the maximal proper ideal   $I^k$ of $W^k(\g,\theta/2)$  is the kernel of the invariant Hermitian form  on $W^k(\g,\theta/2)$, hence one can induce a  invariant Hermitian form  on $W_k(\g,\theta/2)$. The latter vertex algebra is unitary if and only if the invariant form on $W^k(\g,\theta/2)$ is positive semidefinite.
Recall from \cite{AKMPP} that a level $k$ is {\it collapsing} for $W^k(\g,\theta/2)$ if $W_k(\g,\theta/2)$ is contained in its affine vertex algebra part.
\begin{theorem}\label{7a} Assume that $W_k(\g,\theta/2)$ is unitary.
\begin{enumerate}
\item If $\g\ne sl(2)$ is a Lie algebra  then $k$ is a collapsing level.
\item If $\g^\natural$ is not a Lie algebra  then $k$ is a collapsing level.
\end{enumerate}
In particular, if $W_k(\g,\theta/2)$ unitary for three different values of $k$, then either $\g=sl(2)$ or $\g$ is not a Lie algebra and $\g^\natural$ is a Lie algebra.
\end{theorem}
\begin{proof} (1). By assumption $\g_{-1/2}\ne 0$, take a nonzero $u\in  \g_{-1/2}$ such that $\phi(u)=u$ and compute using \eqref{formulona} with $m_1=m'_1=1$:
$$(G^{\{u\}},G^{\{u\}})=(G^{\{u\}}_{-3/2}\vac,G^{\{u\}}_{-3/2}\vac)=\sqrt{-1}\left\langle G^{\{u\}}_{3/2}G^{\{u\}}_{-3/2}\vac\right\rangle=4p(k)\langle u,u\rangle=0.$$If the form on $W^k(\g,\theta/2)$ is positive semidefinite then $G^{\{u\}}\in I^k$, hence $k$ is a collapsing level.

(2).
Take $a\in \g^\natural $ such that $p(a)=1$, $\phi(a)=a$. Compute using \eqref{formulona} with $m_1=m'_1=1$
\begin{align*}
(J^{\{a\}},J^{\{a\}})=(J^{\{a\}}_{-1}\vac,J^{\{a\}}_{-1}\vac)=\sqrt{-1}\left\langle J^{\{a\}}_{1}J^{\{a\}}_{-1}\vac\right\rangle=0,
\end{align*}
hence  $J^{\{a\}}\in I^k$. Assume that $\g^\natural$ is simple;  since $I^k\cap\g^\natural$ is and ideal of $\g^\natural$, then $\g^\natural \subset I^k$. Since 
$\g_{-1/2}$ is not the trivial representation of $\g^\natural$,  there exist  $b\in \g^\natural$ and $u\in \g_{-1/2}$ such that $[b,u]\ne 0$. Since 
$[{J^{\{b\}}}_\l G^{\{u\}}]=G^{\{[b,u]\}}$, \cite[Prop. 3.2]{AKMPP} implies that $k$ is collapsing. The only remaining case, according to \cite[Table 3]{AKMPP}, is
$\g=osp(m|n), m\ge 5$. In this case $\g^\natural=osp(m-4|n)\oplus sl(2)$ and  $\g_{-1/2}=\C^{m-4|n}\otimes \C^2$, and the previous argument applies to $osp(m-4|n)$ acting on 
$\C^{m-4|n}$.
\end{proof}

\begin{remark}
The proof of Theorem \ref{7a} shows more generally that if
there exists an odd (resp. even) element of integer (resp. half-integer) conformal weight in a $W$-algebra $W^k(\g,x)$, which does not lie in the kernel of
its homomorphism to $W_k(\g,f)$, then the latter $W$-algebra is not unitary. 
\end{remark}

In general, even at collapsing levels, the simple vertex algebra $W_k(\g,\theta/2)$ might not be unitary.  It is clear that if $W^k(\g,\theta/2)$ collapses to $\C$ then $W_k(\g,\theta/2)$ is unitary. The list of such cases is given in Proposition 3.4 of \cite{AKMPP}.

 In the next proposition we deal with  other  collapsing levels allowing unitarity.
\begin{prop}\label{7b} Assume $W_k(\g,\theta/2)\ne \C$. If $k$ is a collapsing level and there is a conjugate linear involution $\phi$ on  $W_k(\g,\theta/2)$ such that the corresponding  $\phi$--invariant form is unitary, then the pair $(\g,k)$ is one in the following list
\begin{align}&\g=sl(m|n),\ m\neq n,n+1,n+2, m\ge2, &&k=-1,\label{1}\\
&\g=G_2,&&k=-4/3,\label{2}\\
&\g=osp(m|n),\,\,m-n\geq 10,\ m-n\text{ even,} &&k=-2,\label{3}\\
&\g=spo(2|3),&&k=-3/4,\label{5}\\
&\g=D(2,1;-\frac{1+n}{n+2}),\,n\in \mathbb N,&&k=-\frac{1+n}{n+2}.\label{16}
\end{align}
\end{prop}
\begin{proof}  Looking at 
\cite[Table 5]{AKMPP1} one gets that in the cases listed in the statement there is a conjugate linear involution $\phi$ such that   the $\phi$--invariant Hermitian form on $W_k(\g,\theta/2)$ is positive definite. In   case \eqref{1}  $W_k(\g,\theta/2)$ is $M(\C)$ (Heisenberg vertex algebra) and its unitarity  is shown in  Subsection \ref{bosons}. In cases \eqref{2}, \eqref{3}, \eqref{5}, \eqref{16}, 
$W_k(\g,\theta/2)$ is a simple  affine vertex algebra at positive integral level, hence unitarity follows 
from  Subsection \ref{Affine}.

It remains only to check that the cases in the statement are the only cases where one can have unitarity at a collapsing level $k$, but,  as explained  in  the discussion at the end of Subsection \ref{Affine}, a simple affine vertex algebra $V_k(\g)$ can be unitary if and only if $\g$ is even and $k$ is a positive integer.
\end{proof}
\begin{cor}\label{coro} The following simple minimal W-algebras are unitary:
\begin{enumerate}
\item $W_{-1}(sl(m|n), \theta/2)\cong M(\C),\,m\neq n,n+1,n+2, m\ge2,$ where $M(\C)$ is the Heisenberg vertex algebra  with central charge $c=1$;
\item $W_{-4/3}(G_2,\theta/2)\cong V_1(sl(2))$ with central charge $c=1$;
\item $W_{-2}(osp(m|n), \theta/2)\cong  V_{\frac{m-n-8}{2}}(sl(2)), m-n\geq 10, \text{$m$ and $n$ even}$,  with central charge  \newline $c=\frac{3(m-n-8)}{m-n-4}$;
\item $W_{-3/4}(spo(2|3), \theta/2)\cong V_1(sl(2))$ with central charge $c=1$;
\item $W_{-\frac{1+n}{n+2}}(D(2,1;-\frac{1+n}{n+2}), \theta/2)\cong V_n(sl(2))$ with central charge $c=\frac{3n}{2+n}$, $n\in\ZZ_+$.
\end{enumerate}
\end{cor}
\begin{remark} Case (4)
 of Corollary \ref{coro}  is of special interest since $W_{k}(spo(2|3))$, tensored with one fermion, is the $N=3$ superconformal algebra.
The collapsing level corresponds to the central charge 
$1$ of the simple W-algebra,
isomorphic to $V_1(sl(2))$, hence to the central charge $c=3/2$ of the $N=3$
superconformal algebra, which is therefore unitary. This has been already 
observed in  \cite{SW}.
\end{remark}
\begin{remark} Another interesting case of Corollary \ref{coro}  is (5). Recall 
that $W_k(D(2,1; a)$, tensored with four fermions and one boson, is the big $N=4$
superconformal algebra \cite{KW1}. It follows from Corollary \ref{coro} that this algebra is unitary 
when $a=-\frac{1+n}{n+2}, n\in \mathbb Z_+$, the central charge being $-6a$.
\end{remark}
\section{Appendix: $\l$-brackets in minimal $W$--algebras}\label{6}

If $u\in\g_{-1/2}$ and $v\in \g_{1/2}$, then a direct computation shows that 
$$
[ u,v]=\sum_\a([ u,v]|a^\a)a_\a+\tfrac{([ u,v]|x)}{(x|x)}x=\sum_\a(a_\a|[ u,v])a^\a+\tfrac{(x|[ u,v])}{(x|x)}x,
$$
so
\begin{align*}
[u_\gamma,v]^\natural&=\sum_\a([u_\gamma,v]|a^\a)a_\a=\sum_\a(u_\gamma|[v,a^\a])a_\a,\\
[u,u^\gamma]^\natural&=\sum_\a(a_\a|[u,u^\gamma])a^\a=\sum_\a([a_\a,u]|u^\gamma)a^\a.
\end{align*}
Moreover,
$$
[[ u,u^{\gamma}],[u_{\gamma},v]]^{\natural}=\sum_{\a,\be}([a_\a,u]|u^\gamma)(u_\gamma|[v,a^\be])[a^\a,a_\be].
$$
Since, if $v\in\g_{-1/2}$, $v=\sum_\gamma(v|u^\gamma)[e_{-\theta},u_\gamma]$, we obtain
$$2[e_{\theta},v]=2\sum_\gamma(v|u^\gamma)[e_{\theta},[e_{-\theta},u_\gamma]]=2\sum_\gamma(v|u^\gamma)[x,u_\gamma]=\sum_\gamma(v|u^\gamma)u_\gamma.$$
Substituting we find
\begin{align*}
&\sum_\gamma([a_\a,u]|u^\gamma)(u_\gamma|[v,a^\be])=(\sum_\gamma ([a_\a,u]|u^\gamma)u_\gamma | [v,a^\be])\\
&=2([e_{\theta},[a_\a,u]]| [v,a^\be])=2\langle[a_\a,u],[v,a^\be]\rangle.
\end{align*}
Recall from \cite{AKMPP},  \cite{KMP} that
\begin{align}\label{GGsimplified}
&[{G^{\{u\}}}_{\lambda}G^{\{v\}}]=-2(k+h^\vee)\langle u,v\rangle L+\langle u,v\rangle\sum_{\alpha=1}^{\dim \g^\natural} 
:J^{\{a^\alpha\}}J^{\{a_\alpha\}}:+\\\notag
&\sum_{\gamma=1}^{\dim\g_{1/2}}:J^{\{[u,u^{\gamma}]^\natural\}}J^{\{[u_\gamma,v]^\natural\}}:
 +2(k+1) (\partial+ 2\lambda) J^{\{[[e_{\theta},u],v]^{\natural}\}}\\\notag
&+  \lambda \sum_{\gamma \in S_{1/2}}
  J^{\{ [[ u,u^{\gamma}],[u_{\gamma},v]]^{\natural}\}}
        + 2 p(k)\l^2\langle u,v\rangle,
\end{align}
 where $p(k)$ is a monic quadratic polynomial in $k$,
listed in \cite[Table 4]{AKMPP}. Using the above formulas we can rewrite \eqref{GGsimplified} as
\begin{align}\label{GGsimplifiedfurther}
&[{G^{\{u\}}}_{\lambda}G^{\{v\}}]=-2(k+h^\vee)\langle u,v\rangle L+\langle u,v\rangle\sum_{\alpha=1}^{\dim \g^\natural} 
:J^{\{a^\alpha\}}J^{\{a_\alpha\}}:+\\\notag
&2\sum_{\a,\be}\langle[a_\a,u],[v,a^\be]\rangle:J^{\{a^\a\}}J^{\{a_\be\}}:
 +2(k+1) (\partial+ 2\lambda) J^{\{[[e_{\theta},u],v]^{\natural}\}}\\\notag
&+ 2 \lambda \sum_{\a,\be}\langle[a_\a,u],[v,a^\be]\rangle
  J^{\{ [a^\a,a_\be]\}}
        + 2 p(k)\l^2\langle u,v\rangle.
\end{align}

\vskip20pt
    \footnotesize{

\noindent{\bf V.K.}: Department of Mathematics, MIT, 77
Mass. Ave, Cambridge, MA 02139;\newline
{\tt kac@math.mit.edu}
\vskip5pt
\noindent{\bf P.MF.}: Politecnico di Milano, Polo regionale di Como,
Via Anzani 4, 22100, Como, Italy;\newline {\tt pierluigi.moseneder@polimi.it}
\vskip5pt
\noindent{\bf P.P.}: Dipartimento di Matematica, Sapienza Universit\`a di Roma, P.le A. Moro 2,
00185, Roma, Italy;\newline {\tt papi@mat.uniroma1.it}
}

   \end{document}